 \documentclass[12pt]{amsart}

\usepackage{mathtools}

\mathtoolsset{showonlyrefs=true}

\usepackage[a4paper,left=35mm,right=35mm,top=30mm,bottom=30mm,marginpar=25mm]{geometry}

\usepackage{amsmath}
\usepackage{amssymb}
\usepackage{amsthm}
\usepackage{eurosym}
\usepackage[dvips]{graphics}
\usepackage{graphicx}
\usepackage{epsfig}
\usepackage[normalem]{ulem}
\usepackage[colorlinks]{hyperref}
\usepackage{subcaption}

\usepackage{dsfont}
\usepackage[nocompress, space]{cite}

\usepackage[displaymath,mathlines]{lineno}

\allowdisplaybreaks

\usepackage{hyperref}

\usepackage{ifthen}

\makeindex

\newcommand{\supp}{\operatorname{supp}}

\renewcommand{\div}{\operatorname{div}}

\newcommand{\Rr}{{\mathbb{R}}}

\def\d{{\rm d}}

\def\dy{{\rm d}y}

\def\dt{{\rm d}t}
\def\leq{\leqslant}
\def\geq{\geqslant}

\numberwithin{equation}{section}

\newtheoremstyle{thmlemcorr}{10pt}{10pt}{\itshape}{}{\bfseries}{.}{10pt}{{\thmname{#1}\thmnumber{
#2}\thmnote{ (#3)}}}
\newtheoremstyle{thmlemcorr*}{10pt}{10pt}{\itshape}{}{\bfseries}{.}\newline{{\thmname{#1}\thmnumber{
#2}\thmnote{ (#3)}}}
\newtheoremstyle{defi}{10pt}{10pt}{\itshape}{}{\bfseries}{.}{10pt}{{\thmname{#1}\thmnumber{
#2}\thmnote{ (#3)}}}
\newtheoremstyle{remexample}{10pt}{10pt}{}{}{\bfseries}{.}{10pt}{{\thmname{#1}\thmnumber{
#2}\thmnote{ (#3)}}}
\newtheoremstyle{ass}{10pt}{10pt}{}{}{\bfseries}{.}{10pt}{{\thmname{#1}\thmnumber{
A#2}\thmnote{ (#3)}}}

\theoremstyle{thmlemcorr}
\newtheorem{theorem}{Theorem}
\numberwithin{theorem}{section}

\theoremstyle{thmlemcorr*}
\newtheorem{theorem*}{Theorem}
\newtheorem{lemma*}[theorem]{Lemma}
\newtheorem{corollary*}[theorem]{Corollary}
\newtheorem{proposition*}[theorem]{Proposition}
\newtheorem{problem*}[theorem]{Problem}
\newtheorem{conjecture*}[theorem]{Conjecture}

\theoremstyle{defi}

\newtheorem{hyp}{Assumption}

\theoremstyle{remexample}
\newtheorem{remark}[theorem]{Remark}
\newtheorem{example}[theorem]{Example}

\newtheorem{teo}[theorem]{Theorem}

\newtheorem{pro}[theorem]{Proposition}

\theoremstyle{ass}

\begin{document}
\title[Multi-population Opinion Dynamics Model]{Multi-population Opinion Dynamics Model}
\author{Tigran Bakaryan$^*$}
\address[T. Bakaryan]{
	University of Illinois at Urbana-Champaign, USA}
\email{bakaryt@illinois.edu}
\author{Yuliang Gu$^*$ }
\address[Y. Gu]{ University of Illinois at Urbana-Champaign, USA}
\email{yuliang3@illinois.edu}

\author{Naira Hovakimyan}
\address[N. Hovakimyan]{
	University of Illinois at Urbana-Champaign, USA}
\email{nhovakim@illinois.edu}

\author{Tarek Abdelzaher}
\address[T. Abdelzaher]{
	University of Illinois at Urbana-Champaign, USA}
\email{ zaher@illinois.edu}

\author{Christian Lebiere}
\address[C. Lebiere]{
	Carnegie Mellon University,  Pittsburgh, USA}
\email{cl@cmu.edu}

\def\thefootnote{*}\footnotetext{These authors contributed equally. }\def\thefootnote{\arabic{footnote}}

\keywords{Opinion Dynamics, Bounded Confidence, Multi-Population Model, Mean-Field Limit, System of Continuity Equations.}

\thanks{This work was supported by DOD HQ00342110002.
}

\maketitle

\begin{abstract}

We introduce multi-population opinion dynamics models linked to the bounded confidence model, aiming to explore how interactions between individuals contribute to the emergence of consensus, polarization, or fragmentation.
Existing models either neglect agent similarities, sacrificing accuracy for scalability, or prioritize accuracy by introducing agent-wise connections, constraining scalability. Our proposed model captures similarities between agents in scalable matter.

In our setting, agents similarities are defined by their group affiliations. Specifically, each sub-population is characterized by its distribution, and the closeness between two sub-populations is measured by the Wasserstein distance of their corresponding distributions. 
This leads to two mutually connected dynamics: micro, the individual-based dynamics, and the macro, the distribution-based one.  The individual-wise interactions take into account the population-wise interactions (similarities), and the population-wise interactions are updated based on the individual-wise interactions.

We have proven the well-posedness of our models. Additionally, we conducted several simulations to mimic certain complex social phenomena. 
\end{abstract}

\section{Introduction} 
The advancement of technologies, \cite{MobileSN}, increased the accessibility of online social networks, such as Facebook, X (Twitter), WeChat, and LinkedIn, in our daily lives. Hence, the evolution of opinions within these platforms has become a crucial area of study.  Understanding how individuals influence the decisions of others in social networks has recently emerged as a significant and widely discussed research area, \cite{PredictingSN, Market}. One of the main challenges of working with online social networks (OSNs) is the huge number of agents (users). In this paper, to study the opinion evolution in OSNs, we follow the common approach of "self-organized" models.
In these models, the system's global behavior results from accumulating local interactions. Particularly,  we focus our study on the opinion dynamics models.

Opinion dynamics models aim to describe how opinions form and spread within a population, \cite{degroot1974reaching, deffuant2000mixing, friedkin1990social, rainer2002opinion}. In these models, each individual (or 'agent') of a given population holds an opinion on a particular topic, quantified by a real number. Furthermore, the models provide rules for how these individuals interact with each other and how these interactions influence their opinions.

Before introducing our opinion dynamics models, we briefly overview the bounded confidence (BC) models. The BC models were introduced in \cite{rainer2002opinion, deffuant2000mixing}, and the motivation behind them comes from the concepts of "biased assimilation", \cite{ConfirmBiasedAA}, and "confirmation bias",  \cite{ConfirmBaisAl}. These theories suggest that individuals tend to seek out sources that align with their initial opinions and are more heavily influenced by others whose opinions are close to their own. In other words,  people have tendency to search for information that supports their existing beliefs and ignore the ones that contradict them,  \cite{confbias2022}.
This phenomenon has been widely observed in the Web domain, see \cite{confbiasDataWeb, Whit2013ConfBias} and the references therein. Furthermore, the features of  BC  models are explored in real-world studies of opinion formations see, for example, \cite{RealDataHK, RealDatatw}. Having this, we focus our study on the models related to BC models. 

The bounded confidence models were proposed to understand the opinion formation of the population (to understand macroscopic behavior); that is,  how individual interactions lead to the formation of consensus, polarization, or fragmentation. In the BC models, each individual within a group holds a specific opinion, which the individual may adjust based on the interaction with others, provided the difference between their viewpoints is within a certain threshold (the 'bounded confidence'). Generally, the BC models (for more details see Section \ref{sec:1}) for a given population $\Lambda$ with the population state $\mathbf{X}_t=\{X^i_t\}_ {i\in\Lambda}$ at time $t$ can be  described by the following dynamics
\begin{equation}\label{first}
  \frac{\mathrm{d} X^i_{t} }{\mathrm{d}t}
    = \sum_{j\in\Lambda}f_i(\mathbf{X}_t,\varepsilon_i) X_t^j,
\end{equation}
where $X^i_{t}$ is the state of agent $i$ at time $t$ and the function, $f_i$,   describes the interaction intensity of agent $i$ with other individual within the population. In \eqref{first}, the coefficients  $\varepsilon_i>0$ represent the bounded confidence; that is, agent $i$ does not interact with agents who are outside its $\varepsilon_i$- neighborhood.
Here, the micro-level interactions involve individuals adjusting their opinions based on their peers' opinions. The macro-level behavior that emerges from these interactions may lead to global consensus or polarization, depending on factors such as the initial distribution of opinions, and the size of the confidence threshold of the individuals.

Simplified as they are, BC models still provide a powerful framework for understanding how individual-level cognitive and social processes can lead to macro-level patterns. However, known models often face trade off: either scarify accuracy by simplifying and uniforming the agents interactions rule (i.e. using the same simple $f_i$ in \eqref{first} for all agents) to ensure scalability, see \cite{rainer2002opinion},  or they define agent-wise interaction rule (i.e. employing complex $f_i$ for all agents individually) ensuring accuracy  at the expense of scalability, see \cite{HKsocial-sim}. In this paper, we introduce opinion dynamics models that balance between scalability and accuracy by capturing similarities between agents in a scalable manner. Therefore, the proposed models address the challenge of finding the right equilibrium between scalability and accuracy, empowering our modeling ability. In Section \ref{sec:4}, using our models, we reproduce two interesting social phenomena.

Crucial social and cognitive principles that interface with BC models are homophily \cite{mcpherson2001birds} and cognitive dissonance \cite{festinger1962theory}. The former refers to a tendency of individuals to associate and bond with similar ones, and the latter refers to a psychological state in which a person experiences discomfort or tension due to holding two or more contradictory beliefs. Note that the impact of these principles may manifest beyond the individual level, resulting in intergroup interactions at a larger level, particularly when different groups have incompatible beliefs or values \cite{mcpherson2001birds, festinger1962theory, dandekar2013biased}. In the context of opinion dynamics models, it suggests that individuals may not always update their opinions based purely on the opinions of others. Instead, their "group identities" and the opinions prevalent in their social groups may heavily influence their opinion adjustments. Examples are social groups based on various attributes such as ethnicity, nationality, gender, religion, political affiliation, profession, etc. 

Following the discussion above, we define similarities between agents based on their group affiliations. Particularly,  in our setting each group (population) is described by its distribution and the closeness of two groups is defined by Wasserstein distance between them (see Section \ref{sec:2}). Relaying on this, we introduce population biased or multi-population opinion dynamics models (for more details see Sections \ref{sec:3}-\ref{sec:44}). More precisely, suppose that the population $\Lambda$ is divided into $U$ sub-populations and $\Lambda_i\in U$ is the sub-population of agent $i$. Then, the dynamics in \eqref{first} can be written as follows
\begin{equation}\label{second}
  \frac{\mathrm{d} X^i_{t} }{\mathrm{d}t}
    = \sum_{j\in\Lambda}g_i(\Lambda_i,\Lambda_j)f_i(\mathbf{X}_t,\varepsilon_i) X_t^j,
\end{equation}
where the function, $g_i$, is the population-wise interactions intensity function. Throughout the paper, we consider three main classes of population biased models. The first one is called identity invariant model and studied in Section \ref{sec:3}. In this case, group bias is defined only based on the orientation of groups at the initial time (in \eqref{second} the functions $g_i$ depend only on initial distance between sub-populations in $U$). This model is particularly useful in scenarios where static identifiers, such as race, gender, or long-standing cultural values, significantly influence group dynamics. In Section \ref{sec:44}, we preset the second class, which is called time varying model. Here, the group bias is changing over the time. This model is suitable in scenarios where  groups
adapt or reconfigure in response to changing circumstances, such as political shifts or social upheavals. Furthermore, 
the corresponding mean-field limit is also considered (see Section \ref{sec:inte}), which leads to a system of continuity equations.
The third class of models, we consider in this paper is called multi-identity multi-population model. Section \ref{sec:5} is devoted to study of multi-identity model. In contract to previous two cases, in this model agent may have several identities. That is, we assume that we are given, $V$, partitions  of the population $\Lambda$ and $\Lambda_i^{k}\in U_k$ corresponds to sub-population of agent $i$ with respect to $U_k\in V$ partition. This leads to the following generalization of \eqref{second}
\begin{equation*}
  \frac{\mathrm{d} X^i_{t} }{\mathrm{d}t}
    = \sum_{j\in\Lambda}\Bigg(\sum_{k\in V}g^k_i(\Lambda^{k}_i,\Lambda^{k}_j)\Bigg)f_i(\mathbf{X}_t,\varepsilon_i) X_t^j,
\end{equation*}
where now  $g^k_i$, is the population-wise interactions intensity function with respect to partition $U_k$. Here, the  coefficients $g^k_i$ may be defined according to both identity-invariant model or time-varying model. The multi-identity multi-population model enables us to study complex social phenomena where individuals may have several group affiliations. For example, in the US, people's opinions on the 'Abortion law' are mostly influenced by three major  factors: gender, political orientation, and religion.

Note that in all these models, scalability is ensured due to the definition of similarities between populations based on distributions. 
Furthermore, to enhance the accuracy of the models, the intensity functions (i.e., $f_i$, $g_i$, $g^k_i$) describing interactions between agents or populations are determined through kernel functions (see Section \ref{sec:1}). This approach provides a comprehensive method to define individual or population-wise similarity while preserving the main characteristics of the individual or population.

\section{Local Kernel Model}\label{sec:1}
This section provides insights into BC models and their general form. Furthermore, following the idea of interaction potential in \cite{Blow-upBertozzi_2009, hol2}), we introduce the \textbf{\textit{local kernel}}-based BC model and prove the well-posedness of the model, see Proposition \ref{pro-well-pose-ODE}.

Consider a group/population of $N$ agents, each of which is identified by an index $i \in I = \{1, 2,..., N\}$ and has state (opinion) $X^i_t$ at time $t$ taken value from the state space $\Omega= [-1,1]$. The associated distance between states is the $L_1$ metric on $\Rr$; that is,
\begin{equation}\label{l1}
    d(X^i, X^j) = |X^i - X^j|, \quad i,j \in I.
\end{equation}
We say a subset $\Lambda$ of $I$ is a sub-group/population and denote its collective state by $\mathcal{X}^{\Lambda}_t = \{X_t^i: i \in \Lambda\}$ defined on the product space $\Omega^{|\Lambda|} \subset \Omega^N$. If $\Lambda = I$, $\mathcal{X}^{\Lambda}_t$ stands for the entire group state at time $t$. Each group $\Large$ is described by its distribution, $\mu^{\Lambda}_t \in \mathcal{P}(\Omega)$. In other words, the opinion of the agent $i$,  $X_t^i\in \mathcal{X}^{\Lambda}_t$, of the group $ \mathcal{X}^{\Lambda}_t$ is an independent sample of the group distribution, $\mu^{\Lambda}_t$. Later, we show (see Proposition \ref{pro-well-pose-ODE}) that the opinion profile of the agents is defined on whole $\Rr$ with support on $\Omega$, hence, the distribution, $\mu^{\Lambda}_t$, of the group $\Lambda$ is defined on $\Rr$ and compactly supported on $\Omega$. For our analysis, we use the latest setting.

\begin{remark}
In the context of the opinion dynamics model, opinions are quantified on a scale where $-1$ and $1$ represent the two polar extremes of a given viewpoint. Specifically, a state of $-1$ corresponds to one extreme opinion, while $1$ signifies its direct opposite. Values between these extremes indicate intermediate opinions, capturing the spectrum of viewpoints on the issue at hand. The interpretation of these numerical values heavily depends on the specific topic and context being studied.
\end{remark}

We start our discussion by examining the general framework of the opinion dynamics model. In this model, agents continuously update their opinions, a process that is heavily influenced by their \textit{interactions} with other agents. This dynamic can be mathematically represented as:
\begin{equation}\label{dyn}
  \frac{\mathrm{d} X^i_{t} }{\mathrm{d}t}
    = -\alpha_i X^i_t + \alpha_i f(\mathcal{X}^{I}_t),
\end{equation}
where $f$ is the interaction rule among agents, and $\alpha_i \in [0, 1]$ quantifies the degree of \textit{stubbornness} of each agent, as discussed in various studies \cite{abrahamsson2019opinion, bauso2016opinion}. Many classical opinion dynamics models \cite{degroot1974reaching, friedkin1990social, deffuant2000mixing, rainer2002opinion} suggests that the core mechanism driving these interactions is the \textit{similarity difference (or alignment)} between agents. Specifically, an agent's response is governed by its relative difference from the states of other agents. Within this context, the fundamental structure outlined in Equation~\eqref{dyn} can be expanded into a weighted process as follows:
\begin{equation}\label{general_dyn}
    \frac{\mathrm{d} X^i_{t} }{\mathrm{d}t}
    = -\alpha X^i_t + \alpha \sum_{j=1}^N w_t(X_t^i, X_t^j, \varepsilon)X_t^j,
\end{equation}
where $w_t(\cdot)$ is a time-dependent coefficient quantifying the influence exerted by agent $j$ on agent $i$. Despite the linear appearance of the individual update rule in Equation~\eqref{general_dyn}, the function $w_t(\cdot)$ is capable of introducing complex, non-linear dynamics into the opinion updating process.

\subsection{Local Kernel Methods}
In this section, we present a general framework that integrates the widely-accepted synchronous update rule within a bounded confidence interval. Bounded confidence models, as elaborated in seminal works \cite{rainer2002opinion, deffuant2000mixing}, primarily feature \textit{local} interactions. In these models, an agent's opinion is influenced only by a specific subset of peers, often referred to as its \textit{neighbor set}. This concept is rooted in the theories of bounded rationality \cite{kahneman2003maps} and confirmation bias \cite{nickerson1998confirmation}, reflecting the cognitive and psychological complexities inherent in the opinion-formation process \cite{simon1957models}.

Furthermore, the role of the metric, as defined in Equation~\eqref{l1}, is critical in shaping the opinion space and, consequently, the outcomes of the model. While the use of Euclidean distance is a straightforward and intuitive approach, it often falls short in accurately representing the multifaceted nature of opinion dynamics in real-world scenarios. Particularly, the non-linear nature of social interactions and opinion formation, influenced by factors such as long-term and short-term memory, societal norms, and social identity biases, demands a more nuanced metric.
To better address these complexities, we generalized the BC models based on the interacting potential function (see for example \cite{Blow-upBertozzi_2009, hol1, Convex-hul, hol2}) and introduce the \textbf{\textit{local kernel methods}} to characterize such dynamics. In our setting, the kernel function not only depends on the distance between states (see \eqref{l1}) but also on the state of the considered agent. 

Thus, based on the general framework \eqref{general_dyn}, the update rule for individual agent $i$ is defined by
\begin{equation}\label{general_dyn-2}
\begin{cases}
     \frac{\d X^i_{t} }{\dt}
    =-\alpha_i X^i_t +\frac{\alpha_i}{Z }\sum_{j=1}^N \kappa_d(X_t^i, X_t^j,\varepsilon^i)X_t^j\\
    X^i_0 \sim \mu^i_0,
\end{cases}
\end{equation}
where $\kappa_d: \Rr \times \Rr\times\Rr^+ \to [0,1] $  is a kernel function with the Euclidean distance $d$ \eqref{l1}, $\mu^i_0$ is the initial distribution of the population of agent $i$ with $\supp (\mu^i_0)\subseteq \Omega$, $\varepsilon_{i}$ is a agent-wise confidence threshold defining agent $i$ neighborhood, $\alpha_i$ is level of stubbornness of the agent, and $Z$ is a normalization factor given by
\begin{equation*}
 Z=  Z (X_t^i ,\varepsilon_{i})= \sum_{j=1}^N  \kappa_d(X_t^i ,X_t^j,\varepsilon_{i}).
\end{equation*}
In this particular context of bounded confidence, the kernel may have the following form
\begin{equation}\label{weights}
   \kappa_d(X_t^i, X_t^j,\varepsilon_i)= \begin{cases}  \kappa_d(X_t^i, X_t^j), \quad d(X_t^i , X_t^j) \leq \varepsilon_{i}\\
        0, \quad \quad \quad otherwise.
    \end{cases}
\end{equation}
The kernel employs \textit{locally} within the state space $\Omega$. The kernel method facilitates the modeling of intricate individual cognitive behaviors, especially when dealing with large populations where an agent-by-agent description of interactions is infeasible.

In the subsequent sections, we demonstrate that tailoring kernels based on foundational principles from cognitive science and group dynamics can provide a more adaptive and potent framework. Before that, we shall prove the well-posedness of \eqref{general_dyn-2}; that is, there exists a global solution to the system in \eqref{general_dyn-2}, and each component of the solution stays in $\Omega$.

\begin{pro}\label{pro-well-pose-ODE} Let the kernel $\kappa_d(\cdot,\cdot,\varepsilon) \to [0,1] $ is piece-wise continuous in $\Rr^2$ and $\kappa_d(x,x,\varepsilon)\geq \delta_0 >0$ for all $x\in[-1,1]$. Then, the ODEs system in \eqref{general_dyn-2} has a global solution, $\{X_t^i\}_{i=1}^N$. Moreover, the solution $\{X_t^i\}_{i=1}^N$ belongs to $\Omega^N$.
\end{pro}
\begin{proof} We write the ODEs system in \eqref{general_dyn-2} in following compact form
\begin{equation}\label{dyn-one-eq}
 \begin{cases}
   \frac{\d\bold{X}_t}{\dt}=\bold{F}(\bold{X}_t)
   \\
   \bold{X}_0\sim \bold{\mu}_0,
 \end{cases}
\end{equation}
where $\bold{\mu_0}=(\mu_0^1,\dots,\mu_0^N)$, $\bold{X}_t=(X_t^1,\dots,X_t^N)$ and $\bold{F}(\bold{X}_t)=(F_1(\bold{X}_t),\dots,F_N(\bold{X}_t))$ with
\begin{equation}
    \label{def-Fi}
  \begin{split}
F_i(\bold{X}_t)&=-\alpha_i X^i_t +\alpha_i\sum_{j=1}^N \frac{\kappa_d(X_t^i, X_t^j,\varepsilon_i)}{Z}X_t^j\\&=\alpha_i \sum_{j=1}^N \frac{\kappa_d(X_t^i, X_t^j,\varepsilon_i)}{Z}(X_t^j-X_t^i).
  \end{split}
\end{equation}
Note that the assumptions on the kernel, $\kappa_d$, imply that the function $\bold{F}$ is a piece-wise continuous function. Therefore,  because the system of ODEs in \eqref{dyn-one-eq} is autonomous by \cite[Theorem 1, page 77 ]{Fillipov} and \cite[Point 4, page 81]{Fillipov} it follows that there exists a solution $\bold{X_t}$ to \eqref{dyn-one-eq}. 

Next, we prove that any solution, $\bold{X_t}$, to \eqref{dyn-one-eq} remains within $\Omega^N$. Let $S(t)$ be the convex hull of $\bold{X}_t=(X_t^1,\dots,X_t^N)$ and $s_{min}(t)=\min\{X_t^1,\dots,X_t^N\}$, $s_{max}(t)=\max\{X_t^1,\dots,X_t^N\}$ are extremal points of $S(t)$. Because $S(0)\subset \Omega$ to prove that the solution, $\bold{X_t}$, to \eqref{dyn-one-eq} belongs to $\Omega^N$, it is enough to prove that the set $S(t)$ is non-increasing. This is equivalent to $s_{min}(\cdot)$ being monotone non-decreasing and $s_{max}(\cdot)$ being monotone non-increasing.   

Here, we only prove that $s_{min}(\cdot)$ is monotone non-decreasing, the proof for  $s_{max}(\cdot)$ being monotone non-increasing is similar, so we omit it.  Let at time $t$ $s_{min}(t)=X_t^k$. According to \eqref{dyn-one-eq} and \eqref{def-Fi}       $X_t^k$  solves \begin{equation}\label{ODE-xk}
     \frac{\d X^k_{t}}{\dt} 
    =\alpha_k \sum_{j=1}^N \frac{\kappa_d(X_t^k, X_t^j,\varepsilon_i)}{Z}(X_t^j-X_t^k).
\end{equation}
Recalling that $X_t^i\in \Rr$, we note that  $y\in S(t)$ if and only if $s_{min}(t)\leq y\leq s_{max}(t)$. Hence, $X_t^k=s_{min}(t)\leq X_t^j\in S(t)$. Using this and the fact that $\alpha_k>0$ and $\frac{\kappa_d(X_t^k, X_t^j,\varepsilon_i)}{Z}\geq 0$ from \eqref{ODE-xk} follows that $ \frac{\d X^k_{t}}{\dt}\geq 0$. Thus, $s_{min}(t)$ is non-decreasing. Consequently, any solution, $\bold{X_t}$, to \eqref{dyn-one-eq} stays in $\Omega^N$. 
This  with
 the standard extending solution arguments (see, for example, \cite[Theorem 3.3]{khalil2013nonlinear}) implies the global existence of solutions.
\end{proof}

\begin{remark}
If to the assumptions of Proposition \ref{pro-well-pose-ODE}, we also add locally Lipschitzness of the kernel, $\kappa_d(\cdot,\cdot,\varepsilon)$, then, the solution to \eqref{dyn-one-eq} is also unique.    
\end{remark}

\subsection{Examples of Local Kernel Functions}
To illustrate the versatility of the general framework described by the equation~\eqref{general_dyn-2}, we present several concrete examples of kernel functions. These examples showcase how different kernels can be applied to model varying interaction dynamics:

\begin{itemize}
    \item \textbf{\textit{Uniform kernel}}: This kernel assigns equal weight to all agents within the confidence interval, disregarding the exact distance between opinions.
    \begin{equation}\label{uniform}
       \kappa_d(X_t^i, X_t^j, \varepsilon) = \begin{cases} 
           1, & \text{if } d(X_t^i , X_t^j) \leq \varepsilon,\\
           0, & \text{otherwise}.
       \end{cases}
    \end{equation}

    \item \textbf{\textit{Triangular kernel}}: The influence decreases linearly with the increase in distance between opinions, reaching zero at the edge of the confidence interval.
    \begin{equation*}
       \kappa_d(X_t^i, X_t^j, \varepsilon) = \begin{cases} 
           \varepsilon - d(X_t^i , X_t^j), & \text{if } d(X_t^i , X_t^j) \leq \varepsilon,\\
           0, & \text{otherwise}.
       \end{cases}
    \end{equation*}

    \item \textbf{\textit{Exponential kernel}}: This kernel introduces a non-linear decay of influence with distance, controlled by the parameters $\gamma$ and $\alpha$:
    \begin{equation}\label{expo}
       \kappa_d(X_t^i, X_t^j, \varepsilon) = \begin{cases} 
           \exp \left(- \gamma d(X_t^i , X_t^j)^\alpha \right), & \text{if } d(X_t^i , X_t^j) \leq \varepsilon,\\
           0, & \text{otherwise},
       \end{cases}
    \end{equation}
    where $\gamma, \alpha > 0$ are parameters that modulate the decay rate and smoothness of the influence. Notably, when $\alpha = 2$, it's a squared exponential kernel (a type of \textit{radial basis function kernel}); when $\alpha = 1$, it takes the form of an absolute exponential kernel (\textit{Ornstein–Uhlenbeck kernel}).
\end{itemize}
These kernel functions are integral in modeling the nuanced interactions within the opinion dynamics, allowing for a flexible representation of how agents influence each other based on the proximity of their opinions.
\begin{figure}[h]
\begin{subfigure}{0.8\textwidth}
    \centering
\includegraphics[width=0.8\textwidth]{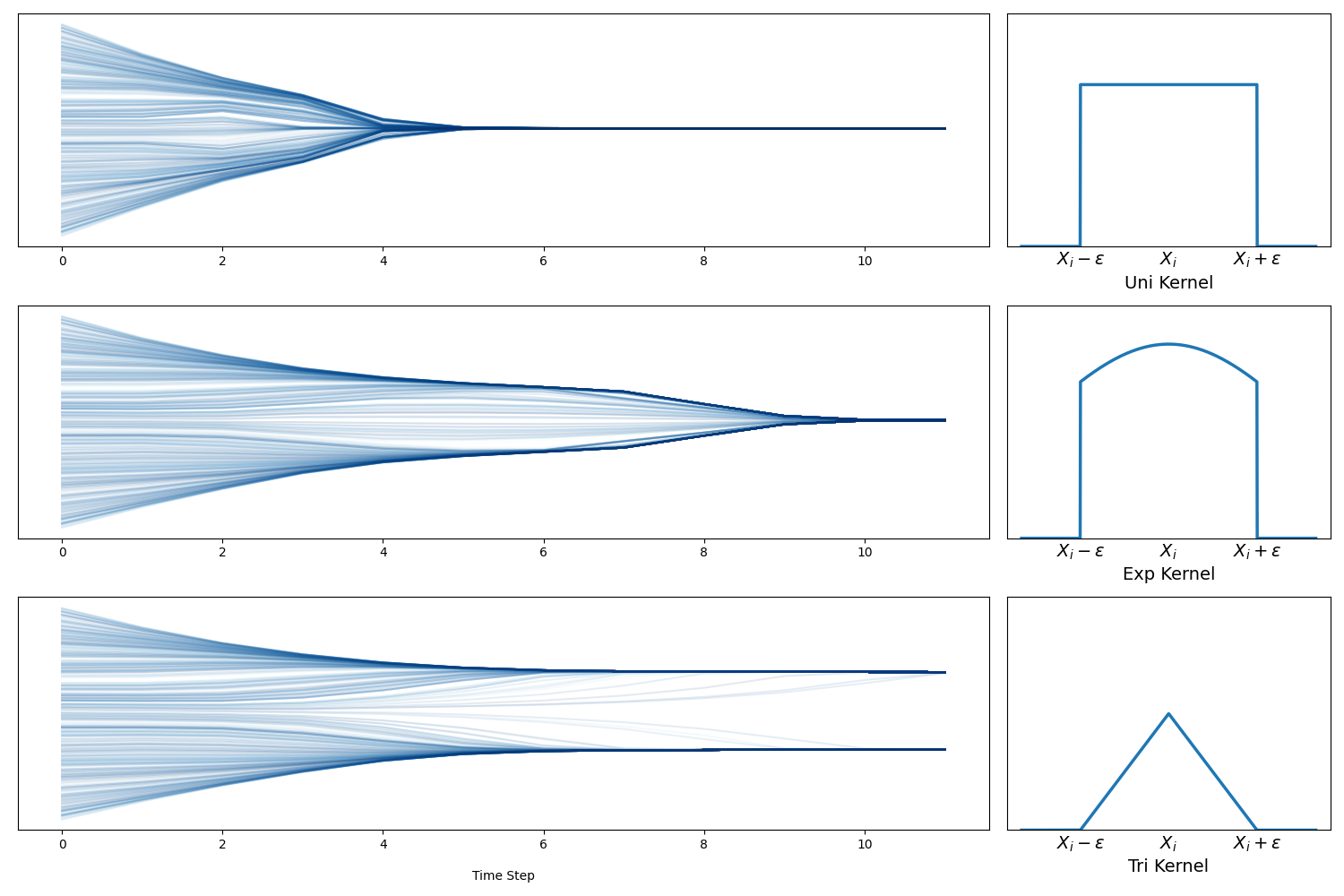}
    \caption{Basic Kernel Functions.}
\label{fig:basic_kernels}
\end{subfigure}
\begin{subfigure}{0.8\textwidth}
    \centering
\includegraphics[width=0.8\textwidth]{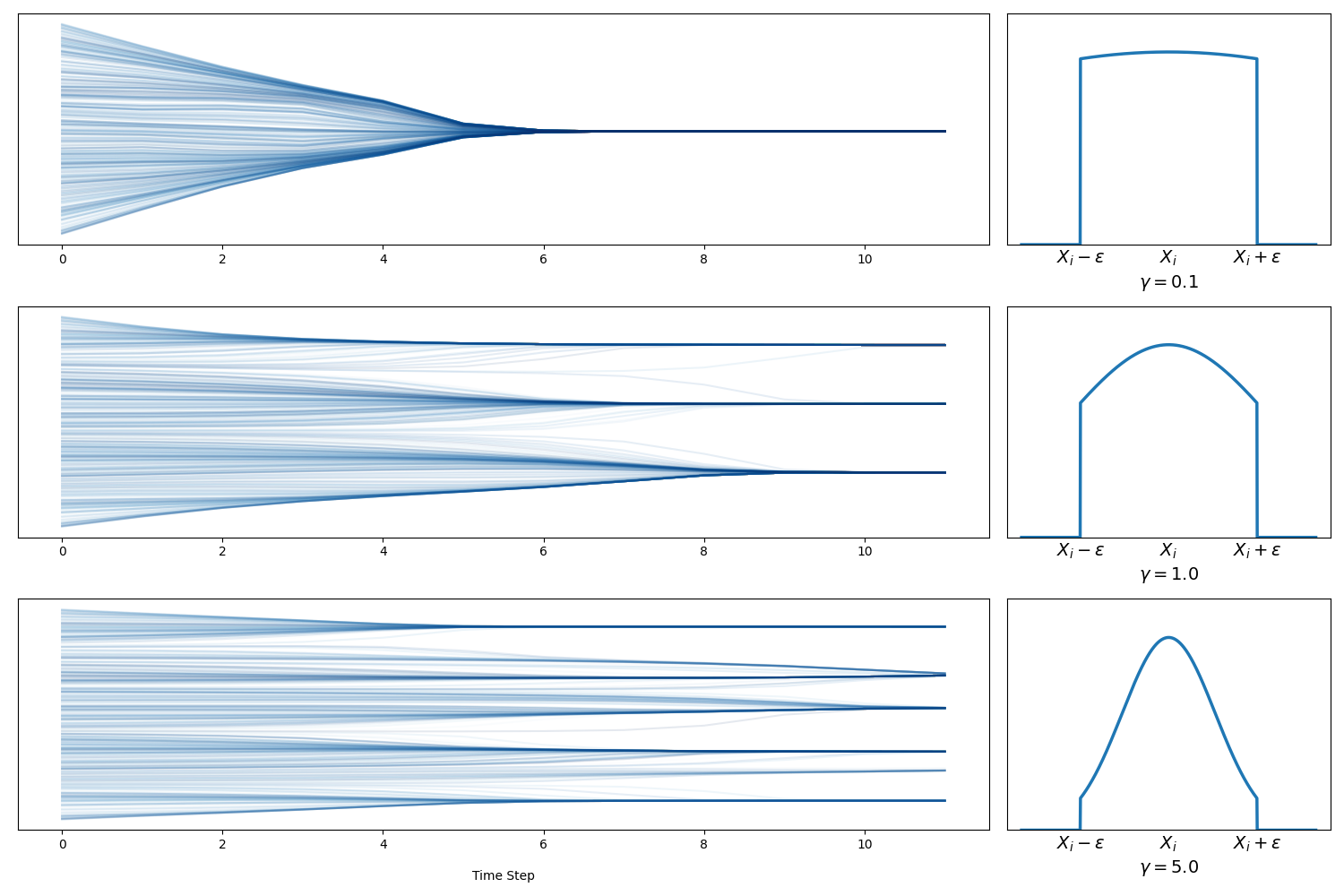}
\caption{Exponential Kernels with Varying $\gamma$.}
\label{fig:exp_kernels}
\end{subfigure}
\caption{Illustration of different local kernel functions: for all simulations, the initial opinion profiles are sampled from a uniform distribution, and model-related parameters, except for the kernels, are kept constant (i.e., bounded confidence $\varepsilon=0.6$, level of stubbornness $\alpha=0.1$, number of agents $N=500$).}
\end{figure}
Figures~\ref{fig:basic_kernels} and \ref{fig:exp_kernels} show the trajectories of opinion profiles over time under the influence of different local kernels. The results illustrate how the \textit{shape} of the kernels affects the collective dynamics of opinions and the resulting patterns in the steady state. The \textit{uniform-like} shape of the kernels (as seen in the first plots of Figures~\ref{fig:basic_kernels} and \ref{fig:exp_kernels}) indicates a higher openness of agents to their neighbors, while a \textit{normal-like} shape suggests greater stubbornness (as seen in the third plots of Figures~\ref{fig:basic_kernels} and \ref{fig:exp_kernels}). A deviation from uniform-like to normal-like shape encourages heterogeneous opinion configurations (opinion fragmentation), because the intensity and range of influence among agents become more localized and selective. Furthermore, the nuanced differences between kernels of similar shapes affect the dynamics, particularly in terms of reaching a steady state. For instance, slight variations in the steepness or spread of a normal-like kernel could either hasten the formation of consensus or deepen the divides, influencing how quickly and to what extent a steady state of opinion distribution is reached

Observe that several kernel functions, notably the exponential kernels, inherently possess a locality characteristic: the intensity of interactions diminishes with increasing distance across the state space. In our model, this local behavior is further emphasized by the parameter $\varepsilon$. This approach not only distinguishes our dynamics from models where such locality is an intrinsic property of the kernel functions but also introduces a \textit{hard cut-off} in interactions. Such a truncation sets the stage for the group-level interaction dynamics, which we shall delve into later in this paper.

\begin{remark}
    If $\kappa_d(X_t^i, X_t^j, \varepsilon^i)$ is a uniform kernel (see equation in ~\eqref{uniform}), the dynamics described by ~\eqref{general_dyn-2} resembles the classical Hegselmann-Krause opinion dynamics model \cite{rainer2002opinion}. The derivation of this relationship is straightforward and is omitted here for brevity.
\end{remark}

\subsection{General Kernel Functions}
Until now, we've restricted our focus on \textit{isotropic} kernel functions, i.e., kernels that depend solely on the distance between agents, i.e. $\kappa_d(X_t^i, X_t^j, \varepsilon) = \kappa(d(X_t^i, X_t^j), \varepsilon)$. These kernels offer several appealing properties: they are inherently simple and intuitive within the context of opinion dynamics, and they maintain the essential character of local interaction, particularly beneficial when dealing with a large number of agents. In many established bounded confidence models of opinion dynamics and their variants, this \textit{isotropy} is a fundamental assumption, primarily to model \textit{Homophily} – the tendency of individuals to connect due to similar characteristics – as observed in social networks \cite{mcpherson2001birds}.

In our formulation (see ~\eqref{weights}), we broaden the scope beyond this isotropic constraint, allowing $\kappa_d(\cdot)$ to be influenced not only by the distance $d(X_t^i, X_t^j)$ but also by the specific states of $X_t^i$ and $X_t^j$, and the parameter $\varepsilon$. This adjustment retains the model's scalability while enhancing its capacity to replicate more complex social dynamics. Next, we will illustrate with a concrete example how this generalization of kernel functions enriches our model, particularly in dynamically capturing nuances of interpersonal attraction.

\begin{example}
\textbf{Resistance in Extreme Opinion Regions.} Research indicates that individuals with extreme opinions tend to be more resistant to change and less open to alternative viewpoints, a phenomenon attributed to \textit{Confirmation Bias} \cite{nickerson1998confirmation} and \textit{Selective Exposure} \cite{stroud2008media}. Conversely, individuals with neutral opinions or those newly exposed to a subject are generally more receptive to persuasion and often experience ambivalence when confronted with conflicting arguments \cite{petty1986elaboration,priester1996gradual}. To model these divergent tendencies in opinion dynamics, we employ the following kernel function:
\begin{equation}\label{combi_kernel_1}
    \kappa_d(X_t^i, X_t^j, \varepsilon) = \begin{cases} 
        \exp \Big(- \gamma(|X_t^i|)d(X_t^i, X_t^j)^\alpha \Big), & \text{if } d(X_t^i, X_t^j) \leq \varepsilon,\\
        0, & \text{otherwise},
    \end{cases}
\end{equation}
where $\gamma(\cdot)$ is a monotonically increasing function of $|X_t^i|$, with $\gamma(0)=0$. Figure~\ref{fig:comb_kernel_1} demonstrates the evolution of $\kappa_d(\cdot)$ with respect to the state over $[-1,1]$, for $\gamma(x)=x$.
\begin{figure}[h]
    \centering
\includegraphics[width=0.8\textwidth]{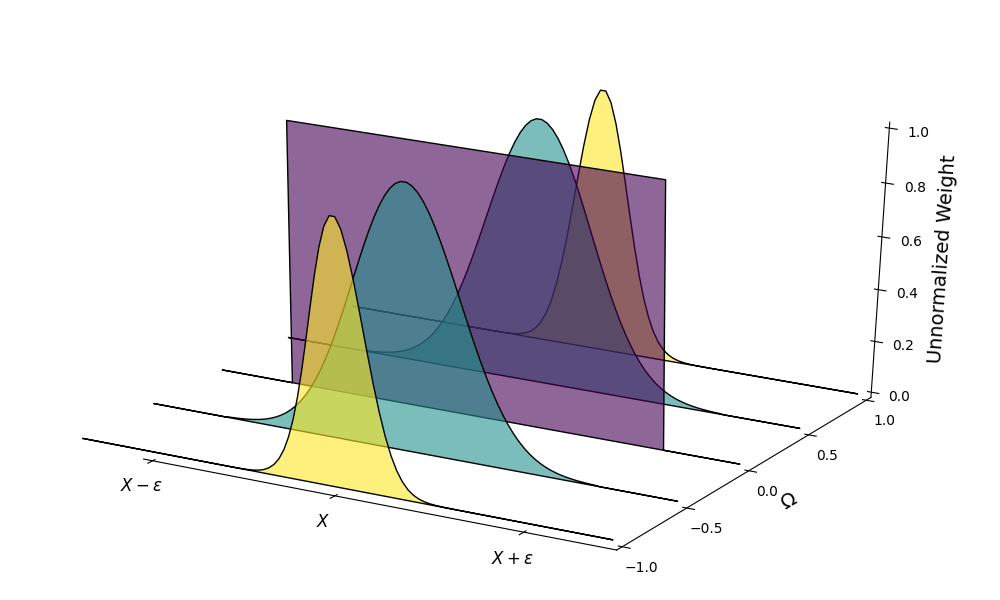}
    \caption{State-Dependent Kernel Function.}
\label{fig:comb_kernel_1}
\end{figure}
Under this kernel, agents with neutral opinions exhibit more openness to their neighbors, akin to uniform weighting within the bounded confidence interval. In contrast, agents with extreme opinions (near $-1$ or $1$) display more stubborn behavior. Figure~\ref{fig:resistance} illustrates the opinion density evolution under such state-dependent kernel functions. Notice the formation of high-density opinion clusters in extreme regions. As agents enter these extreme states due to interpersonal attraction, they tend to remain there due to increased resistance. This tendency facilitates concentration: clusters with more agents exert a strong collective attraction force on others outside the region. Such feedback effects provide a potential explanation for the formation of \textit{Echo Chambers} and public opinion polarization.
\begin{figure}[h]
    \centering
    \includegraphics[width=0.8\textwidth]{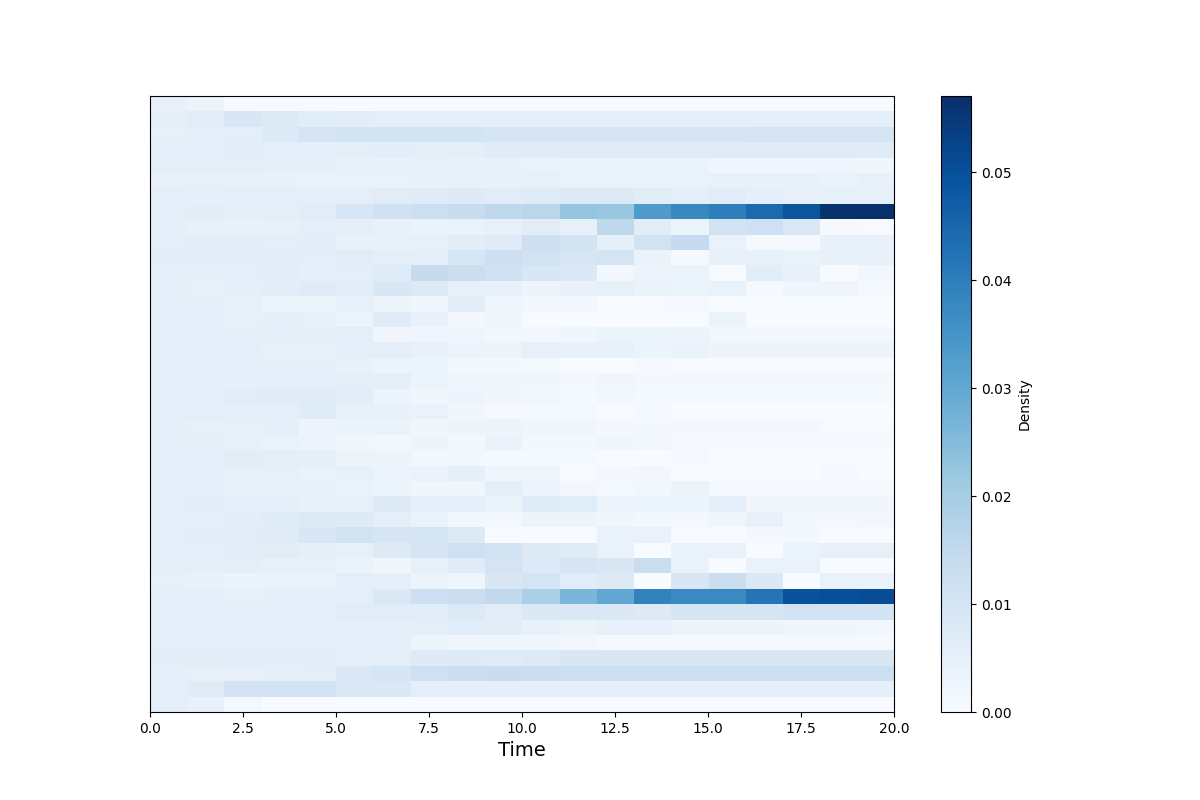}
    \caption{Opinion Density Evolution under State-Dependent Kernel, equation in ~ \eqref{combi_kernel_1}: $\mu_0 = Unif[\Omega]$, $N = 300$, $\alpha = 0.1$, $\varepsilon = 0.6$, $\textit{time steps} = 20$, and $\textit{total number of trials} = 50$. }
    \label{fig:resistance}
\end{figure}
\end{example}

\subsection{Extending to Multi-Population}\label{subsec:multi_p}
The local kernel methods naturally extend to a multi-population (multi-group) setting. Consider a population that can be divided into sub-populations based on their social categories and identities, such as age, gender, race, education level, shared goals, and objectives. In the opinion dynamics setting, the shared characteristics of each sub-population may exhibit a collective interaction mode that differs across various populations. This variation in interaction modes can be captured using different local kernel methods. These local kernels are tailored to capture the specificities of interactions, including the propensity to influence or be influenced, the intensity of interactions, and the nature of information exchange within each group.

Mathematically, the agents set $I$ of size $N$ is partitioned into $p$ \textbf{disjoint} subsets so that each represents a sub-populations; that is,  $P=\{1, 2, ..., p \}$, $I = \cup_{k\in P}\Lambda^k=\cup_{k=1}^{p}\Lambda^k$, $|\Lambda^k|=n_k$  and $\sum_{k=1}^p n_k=N$. Following the discussion above, for agent $i$ at time $t$, we denote its population and corresponding opinion profile by $\Lambda(i)$ and $\mathcal{X}_t^{\Lambda(i)}$. {Using these notations, the general update rule for  agent $i$ is defined by
\begin{equation}\label{general_dyn-3}
\begin{cases}
     \frac{\d X^i_{t}}{\dt} 
    =-\alpha_i X^i_t +\alpha_i\sum_{j=1}^N \frac{\kappa_d^{\Lambda(i)}(X_t^i, X_t^j,\varepsilon_i)}{Z}X_t^j\\
    X^i_0 \sim \mu^i_0,
\end{cases}
\end{equation}
where $\kappa_d^{\Lambda(i)}(\cdot)$ is a population-specific kernel functions. 

In this extended framework, we have preliminarily explored the multi-population setting where the population-dependent bias is \textbf{static}. The group identities do not change over time, resulting in a static interaction mode for each population. Next, we will formally introduce multi-population opinion dynamics models, where the group bias can be dynamic as well and defined concurrently with the opinion formation. This approach helps us study various opinion formation phenomena, such as the formation of echo chambers, the emergence of new groups, etc. Before that, we first establish a proper and uniform metric at the group level in the subsequent section.

\section{Population Similarity Measure}\label{sec:2}
In this section, we define population similarities 
and present the motivation behind them.  
 Furthermore, we illustrate certain properties of Wasserstein distance to motivate our study based on the population-wise similarities given by the Wasserstein metric.

We begin by noting that the distribution of opinions within groups of differing social identities can often reflect the defining characteristics of those identities. For instance, in a survey experiment conducted by \cite{balietti2021reducing}, participants with varying political beliefs expressed their views on the focal issue of wealth redistribution.
The data revealed stark differences in the opinion spectra between self-identified Democrats and Republicans. Specifically, the opinions among Republican participants followed a distribution resembling a normal distribution, and their views, on average, were against wealth redistribution, but did exhibit heterogeneity. In contrast, the opinions among Democrats followed a Gamma-like distribution, with an overwhelming majority supporting wealth redistribution.
This illustrative example underscores how group identities are encoded in the distribution of opinions. It also highlights why using distribution functions provides a more comprehensive and generalizable description compared to relying solely on descriptive statistics such as mean, variance, mode, etc. 

Thus, motivated by the previous example and following the well established theories like particle physics and mean-field games, we quantify the identity of a given set of agents on the state space $\Omega$ by its distribution. More precisely, for a given sub-population  $\Lambda$ with  corresponding states of agents $\mathcal{X}^{\Lambda}_t = \{X_t^i: i \in \Lambda\}$, we describe the collective state or the sub-population $\Lambda$ by  distribution function $\mu^{\Lambda}\in\mathcal{P}(\Rr)$ with  $\supp(\mu^{\Lambda})\subseteq\Omega$. So, the opinion of the agent $i$,  $X_t^i\in \mathcal{X}^{\Lambda}_t$, of the group $ \mathcal{X}^{\Lambda}_t$ is an independent sample of the group distribution, $\mu^{\Lambda}_t$.
Later,  we consider two different  approaches for the determination of  population distribution (see Sections \ref{sec:iso} and \ref{sec:inte}).

 The preceding probabilistic modelling also allows for a mathematical quantification of opinion shifts over time or among different populations. We measure the discrepancies in opinion distributions across various groups by the \textit{statistical distance}. Consider two distinct sets of agents $\Lambda^k$, $\Lambda^q$, their opinion discrepancy \texttt{dist}$(\mathcal{X}^{\Lambda^k}_t, \mathcal{X}^{\Lambda^q}_t)$ is
 \begin{equation*}
    \texttt{dist}(\mathcal{X}^{\Lambda^k}_t, \mathcal{X}^{\Lambda^q}_t) = d(\mu^{\Lambda^k}_t, \mu^{\Lambda^q}_t),
 \end{equation*}
 where $d(\mu^{\Lambda^k}_t, \mu^{\Lambda^q}_t)$ is the statistical distance between probability distributions $\mu^{\Lambda^k}_t, \mu^{\Lambda^q}_t$ at time $t$. 
The notion of statistical distance is used here in a broad sense. Depending on the context, different measures of distance or divergence can be employed, such as Total Variational Distance, $p$-Wasserstein Distance, or $f$-Divergences (a comparison of these distances is demanding  in Figure \ref{dist}). However, for the purposes of our models and analysis, we specifically employ the 1-Wasserstein Distance. We chose this metric because it offers desirable properties from both an interpretive and computational standpoint. \\
\newline
\textbf{Wasserstein Distance as Population-Similarity Measure.}\space Consider two distinct sets of agents $\Lambda^k$, $\Lambda^q$, their distance $d(\mu^{\Lambda^k}_t, \mu^{\Lambda^q}_t)$ is:
\begin{equation}\label{statistical distance}
d(\mu^{\Lambda^k}_t, \mu^{\Lambda^q}_t) = W_p(\mu^{\Lambda^k}_t, \mu^{\Lambda^q}_t),
\end{equation}
where $W_p(\cdot)$ is the Wasserstein distance of order $p$ (also commonly called the \textit{Kantorovich-Rubinstein distance}, when $p=1$). Recalling that we are in one-dimension, the distance has an analytic form 
\begin{equation*}
W_p(\mu^{\Lambda^k}_t, \mu^{\Lambda^q}_t) = \Big (\int_{\Omega} \big|F_{\Lambda^k}(x) - F_{\Lambda^q}(x)\big|^p dx \Big)^{1/p},
\end{equation*}
where $F_{\Lambda^k}$, $F_{\Lambda^q}$ are the cumulative distribution functions (CDF) of  $\mu^{\Lambda^k}_t, \mu^{\Lambda^q}_t$ respectively, for $p=1$ and are quantiles, when $p>1$. 
\begin{figure}[!tbp]
  \centering
  \begin{minipage}[b]{0.45\textwidth}
\includegraphics[width=\textwidth]{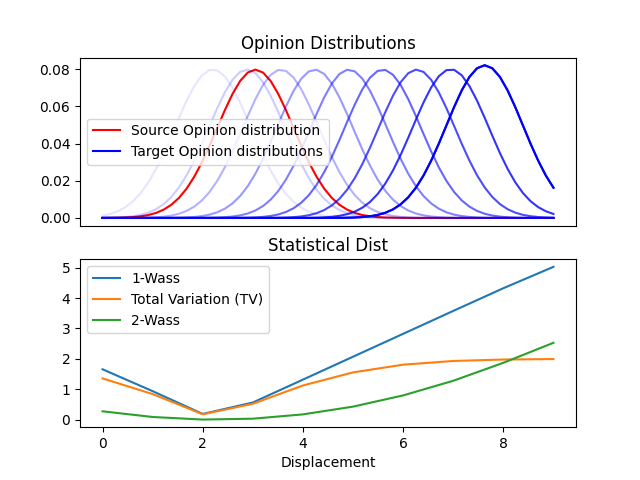}
  \end{minipage}
  \begin{minipage}[b]{0.45\textwidth}
\includegraphics[width=\textwidth]{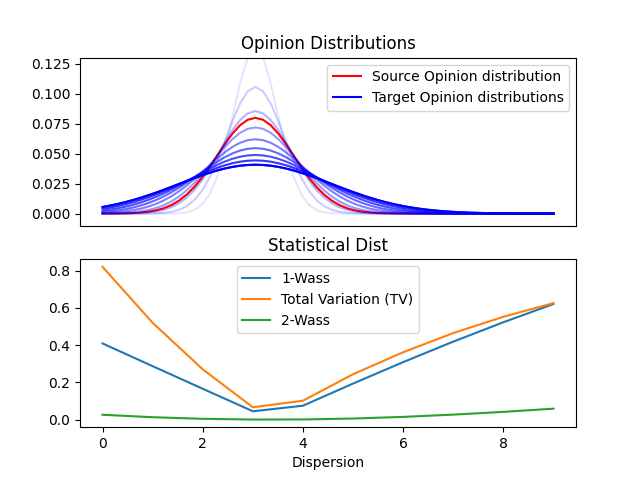}
  \end{minipage}
  \caption{Statistical Distance under Displacement and Dispersion Variation.}
  \label{dist}
\end{figure}
1-Wasserstein distance is well adapted to the opinion dynamics setting with the model we considered \eqref{general_dyn-2} because:
\begin{itemize}
    \item \textit{\textbf{Metric Structure:}} Unlike pseudo-metrics such as divergences, the Wasserstein distance is a true metric. It satisfies important properties like non-negativity, identity of indiscernibles, symmetry, and the triangle inequality. These properties are crucial for quantifying differences in a focal issue between populations in a discernible way. For instance, the identity of indiscernibles ensures that identical opinion spectra have a zero distance between them, providing a baseline reference point for comparisons.
    
    \item \textit{\textbf{Norm Feature:}} The Wasserstein distance of order 1 has a specific \textit{norm} feature (for more details see \cite{villani2009optimal}) that makes it compatible with the absolute distance $d(\cdot)$ in the original state space $\Omega$, i.e. $\|\delta_x - \delta_y\| = d(x,y)$ for all $x,y \in \Omega$, where $\delta_x$ is Dirac delta function. In context, this compatibility allows us to naturally extend distances between individuals to distances between entire groups. This generalization from the individual level to the group level effectively provides an explicit isometric embedding of the original state space $\Omega$ in a Banach space (complete normed vector space), where the dynamics between groups take place. This paves the way for us to define interaction rules on the mean behavior of large group later in the paper.
    
    \item \textit{\textbf{Robustness:}} The Wasserstein distance is stable under mild perturbations of the probability distributions. This is particularly useful when dealing with empirical distributions, which might be subject to noise or sampling errors. In our models where agents have beliefs represented by probability distributions, stability under mild perturbations ensures that small changes in individual beliefs do not lead to disproportionate changes in the overall belief landscape (See Figure \ref{dist}).
\end{itemize}
The Wasserstein distance provides a comprehensive metric for dynamically quantifying the differences in opinion profiles across diverse populations. We have established its suitability for analyzing complex opinion dynamics within varied social identities. In the next sections, we will delve into the formulation of population-level interaction functions grounded in the Wasserstein distance. Our focus will particularly be on two distinct scenarios: \textbf{static group identities}, which reflect consistent societal structures, and \textbf{dynamic group identities}, which capture the evolving nature of social affiliations and perspectives.

\section{Identity-Invariant Multi-Population Model}\label{sec:3}

In this section, we introduce the identity-invariant multi-population model, an advanced extension of the opinion dynamics framework initially presented in Equation~\ref{general_dyn-2} and discussed in Section 2 (see Subsection \ref{subsec:multi_p}). Our model incorporates the dynamics of group identity into the process of opinion formation. It specifically focuses on the probabilistic structural dynamics that are inherent in this phenomenon, providing a more nuanced understanding of how opinions evolve within groups.

The identity-invariant multi-population model is designed to explore the effects of static identifiers such as race, gender, or long-standing cultural values on group opinion dynamics. By integrating these aspects of group identity, the model offers a unique perspective on the interplay between individual opinions and collective identity.

At the core of our model is the application of the Wasserstein distance (Equation~\ref{statistical distance}) as a uniform metric within the collective opinion space. Concretely, we introduce a population-wise local kernel method in the Wasserstein space. This kernel function, \( K \), functions as an interaction functional between sub-populations, defined by their respective probability distribution functions. Formally, the kernel function is defined as \( K: \mathcal{P}_1(\mathbb{R}) \times \mathcal{P}_1(\mathbb{R}) \times \mathbb{R}^+ \to \mathbb{R}^+ \), where \( \mathcal{P}_1(\mathbb{R}) \) represents the space of probability distributions with finite first moments. The function \( K(\mu, \nu, \sigma) = K(W_p(\mu, \nu), \sigma) \) for \( \mu, \nu \in \mathcal{P}_1(\mathbb{R}) \), and \( \sigma \in \mathbb{R}^+ \).

Now, based on the discussions above, we present our first multi-population model, where the population bias has a simple form and it has a straightforward interpretation. We extend the model  in \eqref{general_dyn-3} to the multi-population level:
\begin{equation} \label{multi-p-invar}
    \frac{\d X^i_{t}}{\dt}
    =-\alpha_i X^i_t + \frac{\alpha_i}{Z_\beta} \sum_{\Lambda(j)\in I} \beta_i(t)K(\mu^{\Lambda(i)}_0,\mu^{\Lambda(j)}_0,\sigma_i)\sum_{j \in \Lambda(j)} \kappa_d(X^i_{t},X^j_{t}, \varepsilon_i) X_t^j,
\end{equation}
where the group interaction function $K(\cdot)$ only depends on the initial distributions, $\beta_i$ is decaying factor and $Z_\beta$ is  normalization constant given by 
\begin{equation*}
 Z_\beta=  Z_\beta(\mathbf{\mu}^{\Lambda}_0, \mathbf{X}_t ,\sigma_{i},\varepsilon_{i})= \sum_{j=1}^N  K(\mu^{\Lambda(i)}_0,\mu^{\Lambda(j)}_0,\sigma_i)\kappa_d(X_t^i ,X_t^j,\varepsilon_{i}).
\end{equation*}
In \eqref{multi-p-invar}, the function $K(\cdot)$  defines the interaction between populations (an analog of kernel functions between individuals), and the decaying factor $\beta_i$ describes its intensity change over time. 
The well-definiteness of the ODEs system in \eqref{multi-p-invar} follows from a more general result proved in the next section (see Proposition \ref{pro-well-pose-multi-iso-ODE}).

Contrasting with individual-level dynamics common in single-population models, our framework emphasizes group-level interactions. This distinction is critical as it allows for a more complex interaction function \( K(\cdot) \). The function is engineered to encapsulate a broad spectrum of intergroup relations, including social, demographic, and ideological aspects, thus offering a richer model that mirrors the complexities of real-world dynamics. Next, we give a few concrete examples of group interactions based on the dynamics given in \eqref{multi-p-invar}.
\begin{example}\label{in-group-fav}\textbf{In-Group Favoritism.} In-group favoritism refers to the tendency for individuals to favor or give preferential treatment to members of their own group over those of an out-group\cite{hewstone2002intergroup, everett2015preferences}. This can be captured by the following kernel function
\begin{equation*}
K(\mu^{\Lambda(i)}_0,\mu^{\Lambda(j)}_0,\sigma_i)= 
\begin{cases} \exp (-\Gamma W_p(\mu^{\Lambda(i)}_0,\mu^{\Lambda(j)}_0))
, \quad &W_p(\mu_{\mathcal{X}^i_t},\mu_{\mathcal{X}^j_t}) \leq \sigma_{i}\\
        0,   \quad & \text{otherwise,}
    \end{cases}
\end{equation*}
with the scaling parameter $\Gamma > 0$ that controls the degree of group bias (i.e., a small $\Gamma$ represents a strong preferential bias between two populations, and vice versa). The simulation results in Figure \ref{in group bias} illustrate the impact of in-group bias (plots with decreasing $\Gamma$) on the dynamics of opinion formation over time. Three distinct conditions were simulated: no group bias, moderate group bias, and strong group bias. In the absence of group bias, opinions from different groups displayed considerable convergence, suggesting a scenario where inter-group dialogue might lead to a consensus. As the group bias increased to a moderate level, opinions within each group showed signs of internal alignment, with a clear but diminishing inter-group influence. In the case of strong group bias, the groups' opinions diverged sharply, converging strongly within their respective groups and indicating the formation of echo chambers or polarization.

\end{example}
\begin{figure}[h]
    \centering
\includegraphics[width=0.8\textwidth]{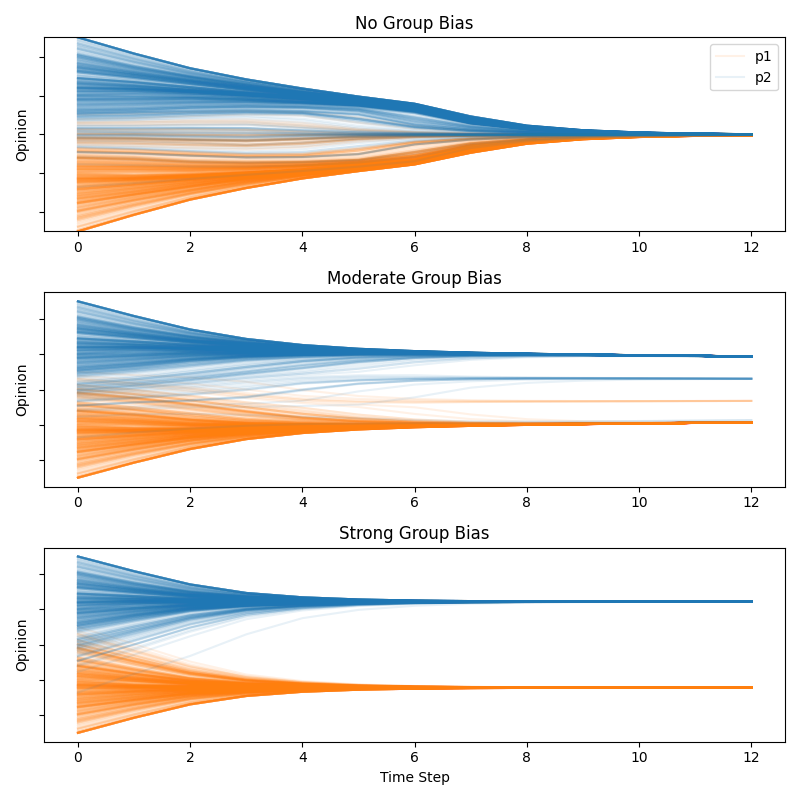}
    \caption{In-Group Bias: initial opinion profiles $\mu_0^{p1}$ and $\mu_0^{p2}$ are sampled from truncated Gaussian distribution, $N=400,\,\ \varepsilon=0.6, \,\ \alpha=0.5,\,\ \kappa_d=uniform\ kernel 
  $, \eqref{uniform}.}
    \label{in group bias}
\end{figure}
\begin{example}\textbf{Asymmetric Group Bias.} The phenomenon of asymmetric group bias is particularly relevant in settings where one group may have a preferential view or favorable bias towards another group, which is not equally reciprocated. To capture this phenomenon in opinion dynamics context, we consider the following kernel function based on the previous example (\ref{in-group-fav}) (for simplicity, a two-population case is studied):
\begin{equation*}
K(\mu^{\Lambda(i)}_0,\mu^{\Lambda(j)}_0,\sigma_i)= 
\begin{cases} \exp (-\Gamma W_p(\mu^{\Lambda(i)}_0,\mu^{\Lambda(j)}_0))
, \quad &W_p(\mu_{\mathcal{X}^i_t},\mu_{\mathcal{X}^j_t}) \leq \sigma_{i}\ \text{and}\ i \in p_1\\
        0,   \quad & \text{otherwise,}
    \end{cases}
\end{equation*}
with the scaling parameter $\Gamma > 0$ that controls the degree of group bias and the condition,  $i\in p_1$, exerts this group bias asymmetrically on the first population. In other words, this bias occurs when one group $p_1$ has a preferential view or favorable bias that is not reciprocated by another group, $p_2$, leading to an imbalanced influence on collective opinion formation. As shown in Figure \ref{asym group bias}, with no bias, opinions are only exchanged within groups, suggesting polarization. With moderate bias, group $p_1$ starts to exert more influence and group $p_2$'s opinions begin to disperse, suggesting a shift towards the dominant group's stance. Under strong bias, group $p_1$'s opinions dominate the narrative, pulling group $p_2$'s opinions towards them. This leads to a significant alignment of group $p_2$'s opinions with that of group $p_1$.
\begin{figure}[h]
    \centering
\includegraphics[width=0.8\textwidth]{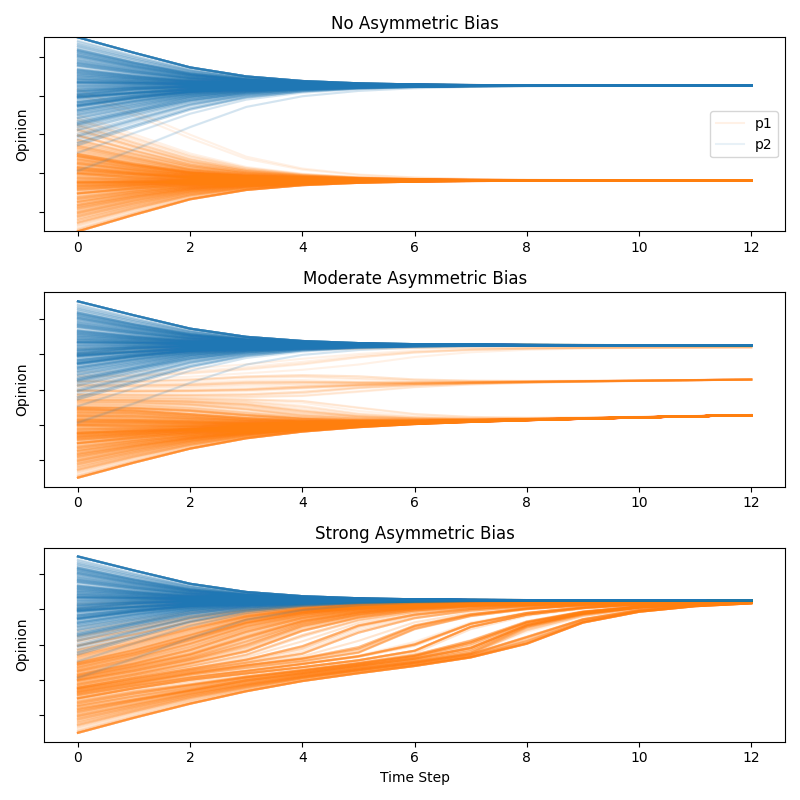}
\caption{Asymmetric Group Bias: $\mu_0^{p1}$ and $\mu_0^{p2}$ are sampled from truncated Gaussian distribution, $N=400, \,\ \varepsilon=0.6,\,\ 
 \alpha=0.5, 
\,\  \kappa_d=uniform \ kernel$, \eqref{uniform}.}
    \label{asym group bias}
\end{figure}
\end{example}
In this section, we have examined the static nature of group opinion dynamics through the lens of the identity-invariant multi-population model. Our exploration encompassed scenarios with both in-group favoritism and asymmetric group bias, offering insights into how static identifiers and inherent biases shape the landscape of opinion formation within and across various groups. The application of the Wasserstein distance and the introduction of a population-wise local kernel method have enabled us to quantify and visualize the nuances of these interactions in a structured and scalable manner.

In the upcoming section, we aim to explore how groups adapt or reconfigure in response to changing circumstances, such as political shifts or social upheavals. This shift in focus will allow us to extend our current model to dynamically evolving environments, where group identities and opinions are not fixed but are fluid and responsive to external stimuli.

\section{Time-Variant Multi-Population Model}\label{sec:44}
In this section, we introduce time-varying multi-population models.
In these models, each agent from every sub-population updates its opinion according to a given rule (dynamics).  Because of this, the collective state of each sub-population is updated, and consequently, their distributions are updated as well. Therefore, the interaction between sub-populations changes. Based on this, we define a time-varying multi-population opinion dynamics model, where 'time-varying' refers to the population bias defined by the changing over time distributions. We considered two types of time-variant models. The first one is called the isolated one. In this case, the population distribution is defined by the considered agents, empirical distribution. Subsection \ref{sec:iso} is devoted to the isolated model, where we prove its well-posedness, see Proposition \ref{pro-well-pose-multi-iso-ODE}. The second class of model is called the integrated multi-population model, which is introduced in Subsection \ref{sec:inte}. Here, agents have access to the actual distribution of their population, which is governed by the system of continuity equations. Theorem \ref{main}

Now, following the discussions above, we introduce time-varying multi-population model
\begin{equation}\label{def-multi-dyn}
\begin{cases}
     \frac{\d X^i_{t}}{\dt}
    =-\alpha_i X^i_t +\frac{\alpha_i}{Z_p} \sum_{j=1}^N K(\mu^{\Lambda(i)}_t,\mu^{\Lambda(j)}_t,\sigma_i)\kappa_d(X^i_{t},X^j_{t},\varepsilon_i)X^j_{t}\\
    X^{i}_0 \sim \mu^i_0,
\end{cases}
\end{equation}
where $\mu^{\Lambda}_t$ is the distribution of the sub-population $\Lambda$ at time $t$, $K(\cdot)$ is an interaction function that defines the update rule between populations (an analog of kernel functions between individuals) and $Z_p$ is a normalization factor given by
\begin{equation*}
 Z_p=  Z_p (\mu^{\Lambda(i)}_t, X_t^i ,\sigma_{i},\varepsilon_{i})= \sum_{j=1}^N  K(\mu^{\Lambda(i)}_t,\mu^{\Lambda(j)}_t,\sigma_i)\kappa_d(X_t^i ,X_t^j,\varepsilon_{i}).
\end{equation*}

In this formulation, the opinion of an individual is seen to be influenced by a combination of factors related to both the specific group $\Lambda(i)$ to which they belong to and the broader population.

Observe that the interactions between groups and individuals are decomposed in \eqref{def-multi-dyn}. So, we can rewrite the dynamics in the following form:
\begin{equation} \label{multi-p-decomp}
    \frac{\d X^i_{t}}{\dt}
    =-\alpha X^i_t + \frac{\alpha}{Z_p} \underbrace{\sum_{k=1}^{p} K(\mu^{\Lambda(i)}_t,\mu^{\Lambda^k}_t,\sigma_i)}_{\text{intergroup interaction}} \underbrace{\sum_{j \in \Lambda^k} \kappa_d(X^i_{t},X^j_{t}, \varepsilon_i) X_t^j}_{\text{in-group interaction}}.
\end{equation}
In Sections \ref{sec:iso} and \ref{sec:inte}, we present two different  approaches for the determination of the sub-populations time-varying distributions   in \eqref{def-multi-dyn} or \eqref{multi-p-decomp}. Before going to that specifications, next, we present three different classes of population-wise kernels, which provides some insight into the modeling capabilities the considered model.\\
\newline
\textbf{Examples of Population-wise Kernel Functions.}
Here, we provide examples of population-wise kernel functions with different use of threshold $\sigma$. 
\begin{itemize}
    \item \textit{No  threshold}: In this type of kernels the representatives of different populations always  interact even though the distance between them ($W_p(\mu,\nu)>>1$) is too big. An interesting kernel of this type is
\begin{equation*}
K(\mu,\nu,\sigma)=K(\mu,\nu) =e^{-W_p(\mu,\nu)}. 
    \end{equation*} 
    \item \textit{  Threshold from above:} This type of kernels does not allow interaction between the representatives of the different populations if the distance between them is too big ($W_p(\mu,\nu)>\sigma$). The general form of this type of kernels is
\begin{equation*}
    K(\mu,\nu,\sigma)= \begin{cases} K_d(W_p(\mu,\nu)), \quad& W_p(\mu,\nu) \leq \sigma\\
            0, \quad &\text{otherwise}.
        \end{cases}
    \end{equation*}
    
    \item \textit{Threshold from below:} The following class of kernels prevent the convergence of different populations; that is, when the populations are getting to close,  $W_p(\mu,\nu)<\sigma$, they stop interacting. These kernel are given by
\begin{equation*}
    K(\mu,\nu,\sigma)= \begin{cases} 0, \quad &W_p(\mu,\nu) \leq \sigma\\
    K_d(W_p(\mu,\nu)), \quad &\text{otherwise}.
        \end{cases}
    \end{equation*}
\end{itemize}

Note that when the function $K(\cdot)$ is from the "Threshold from above" class, the threshold, $\sigma$, is intemperate as a population-wise bounded confidence level. The population-wise bounded confidence is a natural generalization of the individual bounded confidence but gives rich and distinct interpretation in the context of group dynamics. It imposes a hard cut-off on the interactions between individuals based on their group identities. An immediate result of such a threshold in the multi-population opinion dynamics is that agents might not be influenced by their close \textit{out-of-group} neighbors. 
Unlike "Threshold from above" the threshold, $\sigma$, for the class of  "Threshold from below" does not have its counterpart in the individual based models. However, its appearance in the population-dependent model is natural.  


\begin{example}\textbf{Group Cohesion.}\ Group cohesion refers to the sense of solidarity, unity, and bonding that members feel toward their group. This can manifest through shared goals, values, or identities\cite{forsyth2018group}. High group cohesion and low group cohesion represent two ends of a spectrum that can significantly influence the functioning, performance, and well-being of a group. In the context of opinion dynamics, groups with weak group cohesion are more susceptible to external opinions because there's less internal pressure to maintain a consistent group viewpoint. These effects can be captured by the following kernel function:
\begin{equation*}
K(\mu^{\Lambda(i)}_0,\mu^{\Lambda(j)}_0,\sigma_i)= 
\begin{cases} \exp (-\Gamma W_p(\mu^{\Lambda(i)},\mu^{\Lambda(j)}))
, \quad &W_p(\mu^{\Lambda(i)},\mu^{\Lambda(j)}) \leq \sigma_{i}\\
        0,   \quad & \text{otherwise,}
    \end{cases}
\end{equation*}
where parameter $\Gamma > 0$ determines the degree of group cohesion. The Wasserstein distance between the opinion distributions of two groups to determine the level of interaction or influence they have on each other, modulated by the parameter $\Gamma$. A smaller $\Gamma$ suggests a stronger cohesive force within the group, making it less likely to be influenced by external opinions as long as the distance between group opinions is less than the threshold.
\begin{figure}[h]
    \centering
\includegraphics[width=0.8\textwidth]{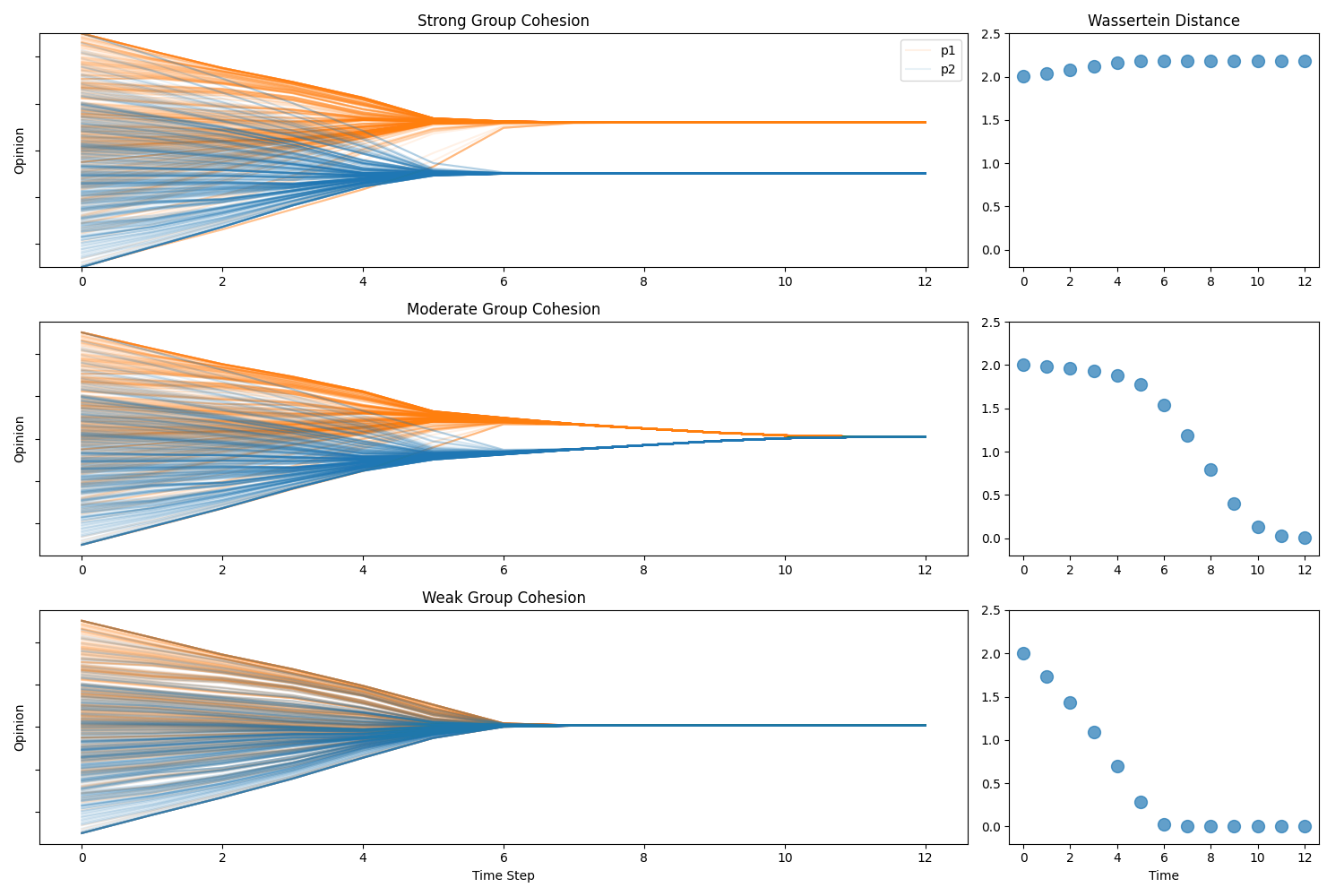}
    \caption{Group Cohesion: initial opinion profiles $\mu_0^{p1}$ and $\mu_0^{p2}$ are sampled from truncated Gaussian distribution, $N=400, \alpha=0.4, \varepsilon=0.4, \kappa_d=exp\ kernel$, see \eqref{expo}.}
    \label{group cohesion}
\end{figure}

The simulation results in Figure \ref{group cohesion} showcase the impact of group cohesion on opinion dynamics. In scenarios of strong group cohesion, opinions remain tightly clustered in group, indicating a group's resistance to external influence and a high degree of internal consensus. For moderate group cohesion, there is a noticeable spread in opinions in the early stage (time step 4 to 6) and consensus in the final stage, suggesting some openness to external ideas while retaining a level of common agreement. In contrast, weak group cohesion leads to straight concensus, reflecting a significant openness to external influences and a lack of group stance. The corresponding Wasserstein Distance plots affirm these observations. Note that this example is different from Example \ref{in-group-fav}, where in-group favoritism was given only based on the initial configurations.  
\end{example}

\subsection{Isolated Multi-Population Model}\label{sec:iso}  
Recall that a given population of  $N$ agents, $I$,  is divided into $p$ sub-populations. That is, for $k = 1, 2, ..., p $, $I = \cup_{k=1}^{p}\Lambda^k$, $|\Lambda^k|=n_k$  and $\sum_{k=1}^p n_k=N$. As before, for agent $i$ at time $t$, we denote its population and corresponding opinion profile by $\Lambda(i)$ and $\mathcal{X}_t^{\Lambda(i)}$. In this subsection, we assume that the sub-population $\Lambda(i)$ is only defined by the agents from the considered population $I$. This means that agents don’t know the actual state of their sub-population (actual distribution),  and they set  their sub-population state  based only on their peers in isolated population, $I$. In other words, the state or distribution of sub-population $\Lambda(i)$ is defined by the empirical distribution of agents' from $\Lambda(i)$. 
Thus, in the multi-population opinion dynamics in \eqref{def-multi-dyn}, population-dependent weights are updated based on empirical distributions. Particularly, the  dynamics is given
\begin{equation}\label{def-multi-dyn-iso}
\begin{cases}
     \frac{\d X^i_{t}}{\dt}
    =-\alpha_i X^i_t +\frac{\alpha_i}{Z_p} \sum_{j=1}^N K_d(\bar{p}_{ij}(t),\sigma_i)\kappa_d(X^i_{t},X^j_{t},\varepsilon_i)X^j_{t}\\
    X^{i}_0 \sim \bar{\mu}^i_0,
\end{cases}
\end{equation}
where 
\begin{equation*}
   \bar{p}_{ij}(t)= W_p(\bar{\mu}^{\Lambda(i)}_t,\bar{\mu}^{\Lambda(j)}_t),
\end{equation*} 
and  
\begin{equation}\label{def-empirical}
\bar{\mu}^{\Lambda(j)}_t(x)=\frac{1}{|\Lambda(i)|}\sum_{X_t^j\in \Lambda(i)}\delta_{X_t^j}(x),
\end{equation}
is empirical distribution of the population $\Lambda(j)$.

\begin{pro}\label{pro-well-pose-multi-iso-ODE} Let $p=1$ and 
$K_d:\Rr^+_0\times\Rr^+_0\to \Rr^+_0$, $\kappa_d:\Rr\times\Rr\times\Rr^+_0\to \Rr^+_0$ be  piece-wise continuous functions. Suppose that there exits $\delta_0>0$ such that $K_d(0,\sigma_i)\geq \delta_0$ and $ \kappa_d(x,x,\varepsilon_i)\geq \delta_0$ for all $x\in[-1,1]$. Then, the ODEs system in \eqref{def-multi-dyn-iso} has a global solution, $\{X_t^i\}_{i=1}^N$, and  $\{X_t^i\}_{i=1}^N \subset \Omega^N$.
\end{pro}
\begin{proof}  
    Note that the cumulative distribution function (CDF) of the sub-population $\Lambda(j)$ is a piece-wise continuous function  
\begin{equation*}
F^{\Lambda(j)}(y)=\frac{1}{|\Lambda(j)|}\sum_{k=1}^{|\Lambda(j)|}\bold{1}_{X_t^k\leq y}.
\end{equation*}
    Because we are in one dimension and $p=1$, we have 
\begin{equation*}
W_1(\bar{\mu}^{\Lambda(i)}_t,\bar{\mu}^{\Lambda(j)}_t)=\int_{\Omega} \big|F^{\Lambda(i)}(y) - F^{\Lambda(j)}(y)\big| \dy. 
\end{equation*}
Therefore, 
\begin{equation*}
    G^{i,j}((X_1,\dots,X_{n_i}),(X_1,\dots,X_{n_j}))=K_d(W_p(\bar{\mu}^{\Lambda(i)}_t,\bar{\mu}^{\Lambda(j)}_t),\sigma_i),
    \end{equation*}
     is a piece-wise continuous function, hence, the right-hand side of \eqref{def-multi-dyn-iso} as well. Thus, arguing as in Proposition \ref{pro-well-pose-ODE}, we complete the proof.
\end{proof}
Note that in \eqref{def-multi-dyn-iso}, we assume that agents are aware of the whole population state; that is, each agent knows the distributions (empirical distributions) of all sub-populations. So, to update their opinions, agents use that information.
However, there are systems where agents do have only local information. Next, in detail, we explain what we mean by local information. Recall that the bounded confidence level, $\varepsilon_i$,  of agent $i$ defines the radius of influence on agent $i$. By a scope of the agent, we understand  the maximal neighborhood of the agent that it has access. Let $s_i\geq\varepsilon_i$ be the scope of agent $i$. Agent $j$ is in the scope of agent $i$ at time $t$ if $|X_t^i-X_t^j|\leq s_i$ and agent $i$ has access to the agent $j$ state $X^j_t$. If in addition $|X_t^i-X_t^j|\leq\varepsilon_i\leq s_i$, then, agent $j$ effect  $i$ and the dynamics is given by
\begin{equation}\label{def-multi-dyn-iso-scope}
\begin{cases}
     \frac{\d X^i_{t}}{\dt}
    =-\alpha_i X^i_t +\frac{\alpha_i}{Z_p} \sum_{j=1}^N K_d(\bar{p}^{s_i}_{ij},\sigma_i)\kappa_d(X^i_{t},X^j_{t},\varepsilon_i)X^j_{t}\\
    X^{i}_0 \sim \bar{\mu}^i_0,
\end{cases}
\end{equation}
where $ \bar{p}^{s_i}_{ij}(t)= W_p(\bar{\mu}^{\Lambda_{s_i}(i)}_t,\bar{\mu}^{\Lambda_{s_i}(j)}_t)$ with $\Lambda_{s_i}(j)$ being the set of all agents from sub-population $\Lambda(j)$ in the scope of agent $i$ and $\bar{\mu}^{\Lambda_{s_i}(j)}_t$ is the empirical distribution of the population $\Lambda(j)$ in the scope agent $i$ or it is the empirical distribution defined by $\Lambda_{s_i}(j)$. Note that Proposition \ref{pro-well-pose-multi-iso-ODE} also holds for \eqref{def-multi-dyn-iso-scope}.


\subsection{Integrated Multi-Population Model} \label{sec:inte} 
Here, we assume that the number of agents is very large (infinity) and to describe the opinion evolution of the population, we follow the mean-field models in \cite{MF-op,MF-2,Tarek} and consider the mean-field limit of  
\eqref{def-multi-dyn-iso}. In contrast to the isolated model, in this case, the agents have access to the actual distribution of their population. That is, each agent is integrated with its population.

We start our discussion with a single population case. For a given single population with $N$ agents, we consider the  dynamics in \eqref{def-multi-dyn} (or \eqref{def-multi-dyn-iso} for a single population)
\begin{equation}\label{def-general_dyn-3}
\begin{cases}
     \frac{\d X^i_{t} }{\dt}
    =-\alpha X^i_t +\frac{\alpha}{Z(X_t^i,\varepsilon)}\sum_{j=1}^N\kappa_d(X_t^i,X_t^j,\varepsilon)X_t^j\\
    X^i_0 \sim \mu_0,
\end{cases}
\end{equation}
with 
\begin{equation*}
   Z (X_t^i ,\varepsilon)= \sum_{j=1}^N  \kappa_d(X_t^i ,X_t^j,\varepsilon).
\end{equation*}
Note that  $\delta_{X_t^i}(X_t^j)=0$ if $i\neq j$ and $\delta_{X_t^i}(X_t^i)=1$. Using this and recalling the definition in \eqref{def-empirical}, heuristically, multiplying and dividing the right-hand sides in \eqref{def-general_dyn-3} by $N$ and letting $N\to \infty$, we get  
\begin{equation}\label{def-mf-dynamics}
\begin{cases}
       \frac{\d X_t}{\dt} =-\alpha X_t+\frac{\alpha }{\varphi(X_t,\mu)}\int_{\Omega} \kappa_d(X_t ,y,\varepsilon)y~\d \mu(y) \\
        X_0\sim \mu_0,
\end{cases}
\end{equation} 
where $\mu(t)$ is distribution of the population and
\begin{equation*}
\varphi(X_t,\mu)= \int_{\Omega} \kappa_d(X_t,\tau,\varepsilon)~\d \mu(\tau).
\end{equation*}
For more details on the limiting argument, see \cite{MP-main, hol2}.

By letting  $N\to \infty$ in \eqref{def-general_dyn-3}, we describe (approximate) the whole population behavior by the representative agent, which dynamics is given by \eqref{def-mf-dynamics}.

Next, we prove the well-posedness of \eqref{def-mf-dynamics} and its solution connection with the solutions of a continuity equation. To do so, we need the following assumption.
\begin{hyp}\label{assump-Lip-ker}
    Let $\varepsilon>0$, $\kappa_d(\cdot,\varepsilon)\in C^1(\Rr^2)\in$ and $\kappa^e_d(x,y,\varepsilon)=\kappa_d(x,y,\varepsilon)y$. Suppose that $\kappa_d(\cdot,\varepsilon)$ and $\kappa^e_d(\cdot,\varepsilon)$ are Lipschitz  in $\Rr^2$. Furthermore,  assume that  there exits $c>0$ such that  $c\leq \kappa_d(x,y,\varepsilon)\leq 1$ for all $x,y\in\Rr$.
\end{hyp}
The following is an example of a class of functions that satisfies the conditions in Assumption \ref{assump-Lip-ker}.   
\begin{example}\label{exam-ass-Lip-ker}
    Let $\varepsilon>0$ and the function $\kappa_d(\cdot,\varepsilon)$ is Lipschitz  in $[-2,2]^2$. Suppose that  there exits $c>0$ such that $c =\kappa_d(x,y,\varepsilon)$ for all $x,y\in \Rr\setminus[-2,2]$ and  $c\leq \kappa_d(x,y,\varepsilon)\leq 1$ for all $x,y\in\Rr$.
\end{example}
Note that the condition that the function, $\kappa^x_d(\cdot)$, is Lipschitz in $\Rr$ (or in Example \ref{exam-ass-Lip-ker} the condition, $c =\kappa_d(x,y,\varepsilon)$ for all $x,y\in \Rr\setminus[-2,2]$) in   Assumption \ref{assump-Lip-ker} is not shrinking the set of choice of the kernel function, $k_d$, because in our setting the models are determined by the values of $k_d$ only defined on $\Omega$. The mentioned condition is only needed for technical reasons.

The well-posedness of \eqref{def-mf-dynamics} is given in the following proposition.

\begin{teo}
Let $T,\varepsilon>0$ and suppose that Assumption \ref{assump-Lip-ker} holds. Then, there exists unique $X_t\in C^1[0,T]$  solving \eqref{def-mf-dynamics}. Moreover, the population distribution, $\mu(t,\cdot)$,  solves the continuity equation
\begin{equation}\label{def-cont-eq}
    \begin{cases}
\mu_t+\div\Big(\Big(-\alpha x+\frac{\alpha}{\varphi(x)} \int_{\Omega} \kappa_d(x,y,\varepsilon)y~\d \mu(t,y)\Big)\mu(t,x)\Big)=0\\
\mu(0,x)=\mu_0,
    \end{cases}
\end{equation}
in the sense of distributions.
\end{teo} provides the existence and uniqueness of solutions to the system of continuity equations and its connection to our integrated opinion dynamics model.

\begin{proof}
Consider the following  equation in whole $\Rr$; that is,
\begin{equation}\label{def-mf-R-dynamics}
\begin{cases}
   \frac{\d X_t}{\dt} =-\alpha X_t+\frac{\alpha }{\varphi_R(X_t,\mu)}\int_{\Rr}\bar{\kappa}_d(X_t,y,\varepsilon)y~\d \mu(y)=b(X_t,\mu) \\
        X_0\sim \mu_0,
\end{cases}
\end{equation} 
where
\begin{equation}\label{def-phiR}    \varphi_R(x,\mu)= \int_{\Rr} \kappa_d(x,\tau,\varepsilon)~\d \mu(\tau),
\end{equation}
and $\bar{k_d}(\cdot,\varepsilon)\in C^1(\Rr^2)$ with
\begin{equation*}
\bar{k}_d(x,y,\varepsilon)= \begin{cases} \kappa_d(x,y,\varepsilon),\quad (x,y) \in [-2,2]^2\\
0,\quad (x,y)\in \Rr^2\setminus[-3,3]^2.
\end{cases}
\end{equation*}
First, we prove that \eqref{def-mf-R-dynamics} and its corresponding continuity equation (see \eqref{def-cont-eq2}) has unique solution. Then, we prove that for the solution the equations in \eqref{def-mf-R-dynamics} and \eqref{def-mf-dynamics} coincides.  
 Denoting 
\begin{equation}\label{def-J}
    J(x,\mu)= \int_{\Rr} \bar{\kappa}_d(X_t,\tau,\varepsilon)\tau~\d \mu(\tau),
\end{equation}
 using the definition of $\bar{k}_d$ and   Assumption \ref{assump-Lip-ker} for $x_1,x_2\in\Rr$ and $\mu_1,\mu_2\in\mathcal{P}_1(\Rr)$, we get    
\begin{equation}\label{proof-ex-mf-one}
\begin{split}
|J(x_1,\mu_1)-J(x_2,\mu_2)|&=\Bigg|\int_{\Rr} \bar{\kappa}_d(x_1,\tau,\varepsilon)\tau~\d \mu_1(\tau)-\int_{\Rr} \bar{\kappa}_d(x_2,\tau,\varepsilon)\tau~\d \mu_2(\tau)\Bigg|\\&\leq\int_{\Rr} |(\bar{\kappa}_d(x_1,\tau,\varepsilon)-\bar{\kappa}_d(x_2,\tau,\varepsilon))\tau|~\d \mu_2(\tau)\\
&+\Bigg|\int_{\Rr}\bar{\kappa}_d(x_1,\tau,\varepsilon)\tau~\d (\mu_1(\tau)-\mu_2(\tau))\Bigg|\\&\leq C\Big(|x_1-x_2|+W_1(\mu_1,\mu_2)\Big),
\end{split} 
\end{equation}
where the constant $C$ does not depend on $x_1,x_2$ and $\mu_1,\mu_2$.  Similarly, we  prove that 
\begin{equation}\label{proof-one-varphi-Lip}
\Big|\varphi_R(x_1,\mu_1)-\varphi_R(x_2,\mu_2)\Big|\leq C\Big(|x_1-x_2|+W_1(\mu_1,\mu_2)\Big).
\end{equation}
Using the preceding estimate, \eqref{proof-ex-mf-one}, Assumption \ref{assump-Lip-ker} and the inequalities $\frac{1}{\varphi_R(x,\mu)}\leq\frac{1}{c}$, $\varphi_R(x,\mu)\leq c$, $|J(x,\mu)|\leq C$, we obtain
\begin{equation*}
\begin{split}
\Bigg|\frac{J(x_1,\mu_1)}{\varphi_R(x_1,\mu_1)}-\frac{J(x_2,\mu_2)}{\varphi_R(x_2,\mu_2)}\Bigg|&\leq\frac{|J(x_1,\mu_1)\varphi_R(x_2,\mu_2)-J(x_2,\mu_2)\varphi_R(x_1,\mu_1)|}{\varphi_R(x_1,\mu_1)\varphi_R(x_2,\mu_2)}\\&\leq \frac{1}{c^2}\Bigg(\Big|\varphi_R(x_2,\mu_2)(J(x_1,\mu_1)-J(x_2,\mu_2))\Big|\Bigg)\\&+\frac{1}{c^2}\Bigg(\Big|J(x_2,\mu_2)(\varphi_R(x_2,\mu_2)-\varphi_R(x_1,\mu_1))\Big|\Bigg)\\&\leq C\Big(|x_1-x_2|+W_1(\mu_1,\mu_2)\Big).
\end{split}
\end{equation*}
Hence, the vector field $b$ is globally Lipschitz; that is, 
\begin{equation}\label{proof-eq5-one}
    (b(x_1,\mu_1)-b(x_2,\mu_2))\leq C \Big(|x_1-x_2|+ W_1(\mu_1,\mu_2)\Big).
\end{equation}
Therefore, by Theorems 2 and 3 in \cite{ContWassEx} the continuity equation corresponding to \eqref{def-mf-R-dynamics}
\begin{equation}\label{def-cont-eq2}
    \begin{cases}
\mu_t+\div\Big(b(x,\mu)\mu(t,x)\Big)=0\\
\mu(0,x)=\mu_0,
    \end{cases}
\end{equation}
has unique solution $\mu(t)\in C([0,T],\mathcal{P}_c^{Ac})$.
 Thus, $b(x,\mu)=\mathbf{b}(t,x)$ is  in time, and by \eqref{proof-eq5-one}   it is globally Lipschitz in $x$ as well. Therefore, by the Picard–Lindel{\"o}f (Carath{\'e}odory) existence and uniqueness result (see Theorem 2.2 in \cite{ODE2012}), the ODE in \eqref{def-mf-R-dynamics} has a unique  solution, $X_t\in C^1[0,T]$. 
 
 Next, we  prove that $X_t\in\Omega$. Note that because $\mu(t)\in\mathcal{P}(\Rr)$, we rewrite  the first equation in \eqref{def-mf-R-dynamics} as follows
\begin{equation}\label{pro-proof-one-eq2}
    \frac{\d X_t}{\dt}=-\frac{\alpha}{\varphi(X_t,\mu)} \Bigg(\int_{\Rr} \bar{\kappa}_d(X_t,\tau,\varepsilon)\Big(X_t-\tau\Big)~\d \mu(\tau)\Bigg).
\end{equation}
Recalling that $\supp(\mu(0))\subseteq\Omega=[-1,1]$,  and using the continuity of the solution to \eqref{def-mf-R-dynamics}, we denote by $0\leq t_e\leq T$ the smallest time for which $X_{t_e}=1$.   If such $t_e$ does not exist or $t_e=T$, the proof is completed. Now, suppose that $0\leq t_e<T$. Note that because the solution to \eqref{def-cont-eq2}, $\mu(t)$, is the push of $\mu(0)$ by $X_t$, then, $\supp(\mu(t))\subseteq \Omega$ for all $t\in [0,t_e]$. Hence, we have
\begin{equation*}
    \int_{\Rr} \bar{\kappa}_d(X_{t_e},\tau,\varepsilon)\Big(X_{t_e}-\tau\Big)~\d \mu(\tau,t_e)=\int_{\Omega}\bar{\kappa}_d(1,\tau,\varepsilon)\Big(1-\tau\Big)~\d \mu(\tau,t_e)\geq 0.
\end{equation*}
 Therefore, from \eqref{pro-proof-one-eq2} it follows that $\dot{X}_{t_e}\leq 0$, which implies that there exists $\Delta t$ such that  $X_t\leq 1$, for all and  $t\in [0,t_e+\Delta t]$.
 Repeating the same arguments finite many times, we conclude that $X_t\leq 1$, for all $t\in [0,T]$. Similarly, we prove that $X_t\geq -1$, for all  $t\in [0,T]$. Thus, we have proved that $\supp(\mu(t))\subseteq\Omega
$. Hence, \eqref{def-mf-R-dynamics} coincides with \eqref{def-mf-dynamics}. This completes the first part of the proof.

Noticing that
 $\mu(t)$ is the push forward of $\mu_0$ by the diffeomorphisms defined by \eqref{def-mf-dynamics} (for more details see Theorem 2.7 in \cite{ContEqSystemEx}), we deduce that $\mu$ solves \eqref{def-cont-eq}.
\end{proof}

Now, we are ready to discuss the multi-population case.
Let there are $P$ sub-populations and assume that the number of agents in each sub-population is very large and they are in the same order as the whole population; that is, there exists $\lambda_1,\dots,\lambda_P \in (0,1)$ such that $\sum_{k=1}^{P}\lambda_k=1$ and $|V^k|=n_k= \lambda_k N$. Therefore, similar to the previous case, by multiplying and dividing the second term in \eqref{multi-p-decomp}   by $N$ and letting $N\to \infty$, we get
\begin{equation}\label{def-multi-mf-dynamics}
\begin{cases} 
\frac{\d X^k_t}{\dt}=-\alpha_k X^k_t+\frac{\alpha_k}{\psi(X_t^k,\mu^k)} \sum_{r=1}^P\lambda_r K_d(p_{kr}(t),\sigma_k)\int_{\Omega} \kappa^k_d(X^k_t,y,\varepsilon_k)y~\d \mu^r(y) \\
        X^k_0\sim \mu^k_0,
\end{cases}
\end{equation} 
where  $ p_{kr}(t)= W_p(\mu^{\Lambda_k}_t,\mu^{\Lambda_r}_t)
$, 
  $\mu^k(t)$ is the distribution of the sub-population $k$  and
\begin{equation}\label{def-Zp-law}
\psi(X^k_t,\mu^k)=\sum_{r=1}^P\lambda_r K_d(\mu^k(t),\mu^r(t),\sigma_k)\int_{\Omega}\kappa^k_d(X^k_t,y,\varepsilon_k)~\d \mu^r(y).
\end{equation}
As in the case of a single population, here, as well, the sub-populations are substituted by their representative agents (e.g. $X_t^k$ is the representative agent of the sub-population $\Lambda^k$) and the system dynamics is governed by \eqref{def-multi-mf-dynamics}.

\begin{hyp}\label{assump-Lip-pop-ker}
    Suppose that the function $K_d(\cdot,\cdot)$ is Lipschitz  in $\Rr^2$ and satisfies  $c\leq K_d(\cdot,\cdot)\leq \frac{1}{c}$ for some $c>0$. 
\end{hyp}

The following result proves the well-posedness of the system in \eqref{def-multi-mf-dynamics} and makes a connection between the solutions of \eqref{def-multi-mf-dynamics} with the system of continuity equations. This system is deterministic and may be used for numerical approximation. 

\begin{teo}
\label{main} 
Let $T>0$ and   $\varepsilon_k, \sigma_k>0$, for $k=1,\dots,P$.  Suppose that Assumption \ref{assump-Lip-pop-ker} holds and for all $k=\{1,\dots,P\}$ the kernels, $\kappa_d^k (\cdot, \varepsilon_k)$, satisfy   Assumptions \ref{assump-Lip-ker}. Then, there exists a unique $\{X^k_t\}_{k=1}^{P}\in (C^1[0,T])^P$ solving system of DEs in \eqref{def-multi-mf-dynamics}. Moreover, the corresponding distributions, $\{\mu^k(t,\cdot)\}_{k=1}^{P}$, of the sub-populations solve the
system of continuity equations   
\begin{equation}\label{def-cont-eq-multi}
\begin{cases}
\mu^k_t+\Big(\Big(-\alpha_k x+\frac{\alpha_k}{\psi(x,\mu^k)} \sum_{r=1}^P\lambda_rK_d(p_{kr}(t),\sigma_k)\int_{\Omega} \kappa_d(x,y,\varepsilon_k)y~\d \mu^r(y) \Big)\mu^k\Big)_x=0\\   \mu^k(0,x)=\mu^k_0,
    \end{cases}
\end{equation}
in the sense of distributions. 
\end{teo}
\begin{proof}  For convenience, we rewrite the system of equations in  \eqref{def-cont-eq-multi} 
\begin{equation}\label{proof-cont-eq-R-multi}
\begin{cases}
\mu^k_t+\Big(V(x,\tilde{\mu})\mu^k\Big)_x=0\\   \mu^k(0,x)=\mu^k_0,
    \end{cases}
\end{equation}
where $\tilde{\mu}=\{\mu^1,\dots,\mu^{P}\}$ and 
\begin{equation}\label{def-V}
 V(x,\tilde{\mu})=  \alpha_k \Bigg(-x+\frac{\sum_{r=1}^P\lambda_rK_d(p_{kr}(t),\sigma_k)J^k(x,\mu^r)}{\sum_{r=1}^P\lambda_rK_d(p_{kr}(t),\sigma_k)\varphi_R^k(x,\mu^r)} \Bigg),
\end{equation}
with $\varphi^k_R$, $J^k$  are defined by  \eqref{def-phiR}, \eqref{def-J}, respectively, for $\kappa_d^k$. 
To prove the existence and uniqueness of solutions to \eqref{proof-cont-eq-R-multi}, we following the proof of Theorem 1.1 in \cite{ContEqSystemEx}.

Note that Assumption \ref{assump-Lip-pop-ker} with the trigonometric inequality in Wasserstein space implies 
\begin{equation*}
\begin{split}
|K_d(p^1_{kr},\sigma_k)-K_d(p^2_{kr},\sigma_k)|&=|K_d(W_1(\mu_1^k,\mu^r_1),\sigma_k)-K_d(W_1(\mu_2^k,\mu^r_2),\sigma_k)|\\ &\leq C|W_1(\mu_1^k,\mu^r_1)-W_1(\mu_2^k,\mu^r_2)|\\&\leq
C(W_1(\mu_1^k,\mu^k_2)+W_1(\mu_1^r,\mu^r_2)).
\end{split}
\end{equation*}
Using this, the uniform bound of the functions $K_d$ and $J^k$  from \eqref{proof-ex-mf-one}, we get 
\begin{equation}\label{proof-multi-eq-Kd}
\begin{split}
        |K_d(p^1_{kr},\sigma_k)J^k(x,\mu^r_1)&-K_d(p^2_{kr},\sigma_k)J^k(x,\mu^r_2)|\\&\leq |K_d(p^1_{kr},\sigma_k)||J^k(x,\mu^r_1)-J^k(x,\mu^r_2)|\\&+
       |J^k(x,\mu^r_2)| |K_d(p^1_{kr},\sigma_k)-K_d(p^2_{kr},\sigma_k)|\\&
       \leq C (W_1(\mu_1^k,\mu^k_2)+W_1(\mu_1^r,\mu^r_2)),
\end{split}
\end{equation}
for all $\tilde{\mu}_1,\tilde{\mu}_2\in (\mathcal{P}(\Rr))^P$. Similarly, using \eqref{proof-one-varphi-Lip} we have
\begin{equation}\label{proof-mmulti-Lip-g}
|K_d(p^1_{kr},\sigma_k)\varphi_R^k(x,\mu^r_1)-K_d(p^2_{kr},\sigma_k)\varphi_R^k(x,\mu^r_2)|\leq C (W_1(\mu_1^k,\mu^k_2)+W_1(\mu_1^r,\mu^r_2)).
\end{equation}
Consequently, recalling that $K_d,\kappa^k_d\geq c$ and the fact that functions $\kappa^k_d$ are Lipschitz in $x$ from the equations in \eqref{def-V}, \eqref{proof-multi-eq-Kd} and \eqref{proof-mmulti-Lip-g}, we deduce that
\begin{equation}\label{proof-mmulti-Lip-V}
|V(x,\tilde{\mu}_1)-V(y,\tilde{\mu}_2|\leq C\Big(|x-y|+\sum_{r=1}^{P} W_1(\mu^r_1,\mu^r_2)\Big),
\end{equation}
for all $x,y\in \Rr$ and $\tilde{\mu}_1,\tilde{\mu}_2\in (\mathcal{P}(\Rr))^P$ .

Let $0<t_0\leq T$ and fix $\tilde{\nu}\in L^\infty([0,t_0],(\mathcal{P}(\Omega))^P)$. Note that  $L^\infty([0,t_0],(\mathcal{P}(\Omega))^P)$ is Banach space with respect to the distance
\begin{equation*}
d(\tilde{\nu}_1(t),\tilde{\nu}_2(t))=\sup_{t\in[0,t_0]}\Big(\sum_{r=1}^{P}W_1(\nu^r_1(t),\nu^r_2(t))\Big).
\end{equation*}

Now, substituting $\tilde{\nu}$ in into \eqref{proof-cont-eq-R-multi}, we consider the following equation 
\begin{equation}\label{proof-cont-eq3}
\begin{cases}
\mu^k_t+\Big(V(x,\tilde{\nu})\mu^k\Big)_x=0\\   \mu^k(0,x)=\mu^k_0.
    \end{cases}
\end{equation}
Note that the equations of the system in \eqref{proof-cont-eq3} are decoupled. Furthermore,  the corresponding vector fields $v(t,x)=V(x,\tilde{\nu}(t))$  are Lipschitz in $x$ (see \eqref{proof-mmulti-Lip-V}), then, 
by \cite[Theorem 5.34]{villani2003topics} the system of equations in \eqref{proof-cont-eq3} has a unique solution, $\{\mu^k\}_{k=1}^{P}\in C([0,t_0],\mathcal{P}^{ac}_c(\Rr))$. 

Next, we prove that $\supp(\mu^k(t))\subseteq \Omega$ for all $t\in [0,t_0]$.
Consider the corresponding flow for \eqref{proof-cont-eq3} 
\begin{equation}\label{pro-proof-eq2}
   \begin{cases}
        \frac{\d X^k_t}{\dt}=-\frac{\alpha_k}{\psi(X_t^k,\nu^k)} \Bigg(\sum_{r=1}^P\lambda_r K_d(q_{kr}(t),\sigma_k)\int_{\Omega} \kappa_d(X_t^k,y,\varepsilon_k)\Big(X_t^k-y\Big)~\d \nu^r(y)\Bigg),
    \\
        X^k_0\sim \mu^k_0,
   \end{cases}
\end{equation}
where $q_{kr}(t)=W_1(\nu^k(t),\nu^r(t))$ and $\psi$ is defined by \eqref{def-Zp-law}.
Recalling that $\supp(\mu^r(0))\subseteq\Omega=[-1,1]$, for all $r=1,\dots,P$,  and using the continuity of the solution to \eqref{pro-proof-eq2}, we denote by $0\leq t_e\leq t_0$ the smallest time for which there exits $k\in\{1,\dots,P\}$ such that $X^k_{t_e}=1$. The proof is completed if such $t_e$ does not exist or $t_e=t_0$. Now, suppose that $0\leq t_e<t_0$, hence, $X^{k}_{t_e}=1$. Consequently, recalling that  $\nu^r(t_e)$ is the push forward of $\nu^r(0)$ with respect $X^k_{t_e}$, we have $\supp(\nu^r(t_e))\subseteq\Omega=[-1,1]$, for all $r=\{1,\dots,P\}$, which yields
\begin{equation*}
    \int_{\Omega} \kappa_d(X_{t_e}^k,y,\varepsilon_k)\Big(X_{t_e}^k-y\Big)~\d \nu^r(y,t_e)=\int_{\Omega} \kappa_d(1,y,\varepsilon_k)\Big(1-y\Big)~\d \nu^r(y,t_e)\geq 0.
\end{equation*}
 Therefore, from \eqref{pro-proof-eq2} it follows that $\dot{X}^k_{t_e}\leq 0$, which implies that  there exists $\Delta t$ such that  $X^k_t\leq 1$, for all  $t\in [0,t_e+\Delta t]$.
 Repeating the same arguments finite time, we conclude that $X^k_t\leq 1$, for all $ k=1,\dots,P$ and  $t\in [0,t_0]$. Similarly, we prove that $X^k_t\geq -1$, for all $ k=1,\dots,P$ and  $t\in [0,t_0]$. Thus,  taking into account that the solution to \eqref{proof-cont-eq3}, $\mu^k(t)$, is the push forward of $\mu_0^k$ through $X^k_t$ and $X^k_t\in\Omega$, we deduce that   $\supp(\mu^k(t))\subseteq\Omega
$ for $ k=1,\dots,P$, hence, $\mu\in L^\infty([0,t_0],(\mathcal{P}(\Omega))^P)$. 

Let $S:L^\infty([0,t_0],(\mathcal{P}(\Omega))^P)\to L^\infty([0,t_0],(\mathcal{P}(\Omega))^P)$ is defined as follows: for all $\tilde{\nu}\in L^\infty([0,t_0],(\mathcal{P}(\Omega))^P)$, $S(\tilde{\nu})=\tilde{\mu}$, where $\tilde{\mu}\in L^\infty([0,t_0],(\mathcal{P}(\Omega))^P)$ is the unique solution to \eqref{proof-cont-eq3}. 

Next, we chose $t_0$ such that the map, $S$, becomes a contraction.
Because in \eqref{proof-cont-eq3} the state variable $x$ belongs to $\Omega$ and $V$ satisfies the Lipschitz condition in \eqref{proof-mmulti-Lip-V} the Proposition 4.2 in \cite{ContEqSystemEx} holds true in our setting as well. Let $\tilde{\nu}_1,\tilde{\nu}_2\in L^\infty([0,t_0],(\mathcal{P}(\Omega))^P)$, $S(\tilde{\nu}_i)=\tilde{\mu}_i$, $i=1,2$. Noticing that $\tilde{\mu}_1$ and $\tilde{\mu}_2$ have the same initial conditions from \cite[Proposition 4.2]{ContEqSystemEx}, we get
\begin{equation}\label{C}
\sup_{t\in[0,t_0]}\Big(\sum_{r=1}^{P}W_1(\nu^r_1(t),\nu^r_2(t))\Big)\leq C t_0 e^{Ct_0} \sup_{t\in[0,t_0]}\Big(\sum_{r=1}^{P}W_1(\mu^r(t),\mu^r_2(t))\Big).
\end{equation}
 Choosing $t_0$ such that $ C t_0 e^{Ct_0} <1$, we obtain that the map  $S$ is a contraction. Therefore, by Banach's fixed point Theorem, we deduce that there exists a unique solution to \eqref{proof-cont-eq-R-multi} in $[0,t_0]$. Finally, note that the constant $C$ in \eqref{C} only depends on problem data, hence, iterating the arguments, we obtain the existence and uniqueness of solutions to \eqref{proof-cont-eq-R-multi} on the whole interval $[0,T]$.

Because the coefficients of the system in \eqref{def-multi-mf-dynamics} are Lipschitz by \cite[Theorem 2.7]{ContEqSystemEx}, the existence and uniqueness of solutions of problems in \eqref{def-multi-mf-dynamics} and \eqref{proof-cont-eq3}.
\end{proof}
 Note that in the dynamics in \eqref{def-multi-mf-dynamics}, the distributions, therefore, population-wise weights, are updated by the system of continuity equations, \eqref{def-cont-eq}. 
Next, we present an example that demonstrates a scenario where the isolated and integrated models lead to different behaviors due to the small number of agents.

\begin{example}\textbf{Isolated and Integrated Models.}\ Let we are given two populations whose distributions are drowned in the first plot of the second row of Figure \ref{integrated}. As one can see the actual distributions of these two populations are quite different, and their Wasserstein distance is quite big. However, taking a small number of agents from each population (taking a small number of samples from the corresponding distributions), it may happen that all considered agents' opinions are close to each other’s; that is, their empirical distinctions are close to each other, see the first plot of the first row of Figure \ref{integrated}. Hence, in the case of the isolated model, \eqref{def-multi-dyn-iso}, these two populations converge, the first row of Figure \ref{integrated}. In contrast, the integrated model, \eqref{def-multi-mf-dynamics}, leads to polarization due to the big difference of the actual distributions,  the second row of Figure \ref{integrated}.  
\begin{figure}[h]
    \centering
\includegraphics[width=0.8\textwidth]{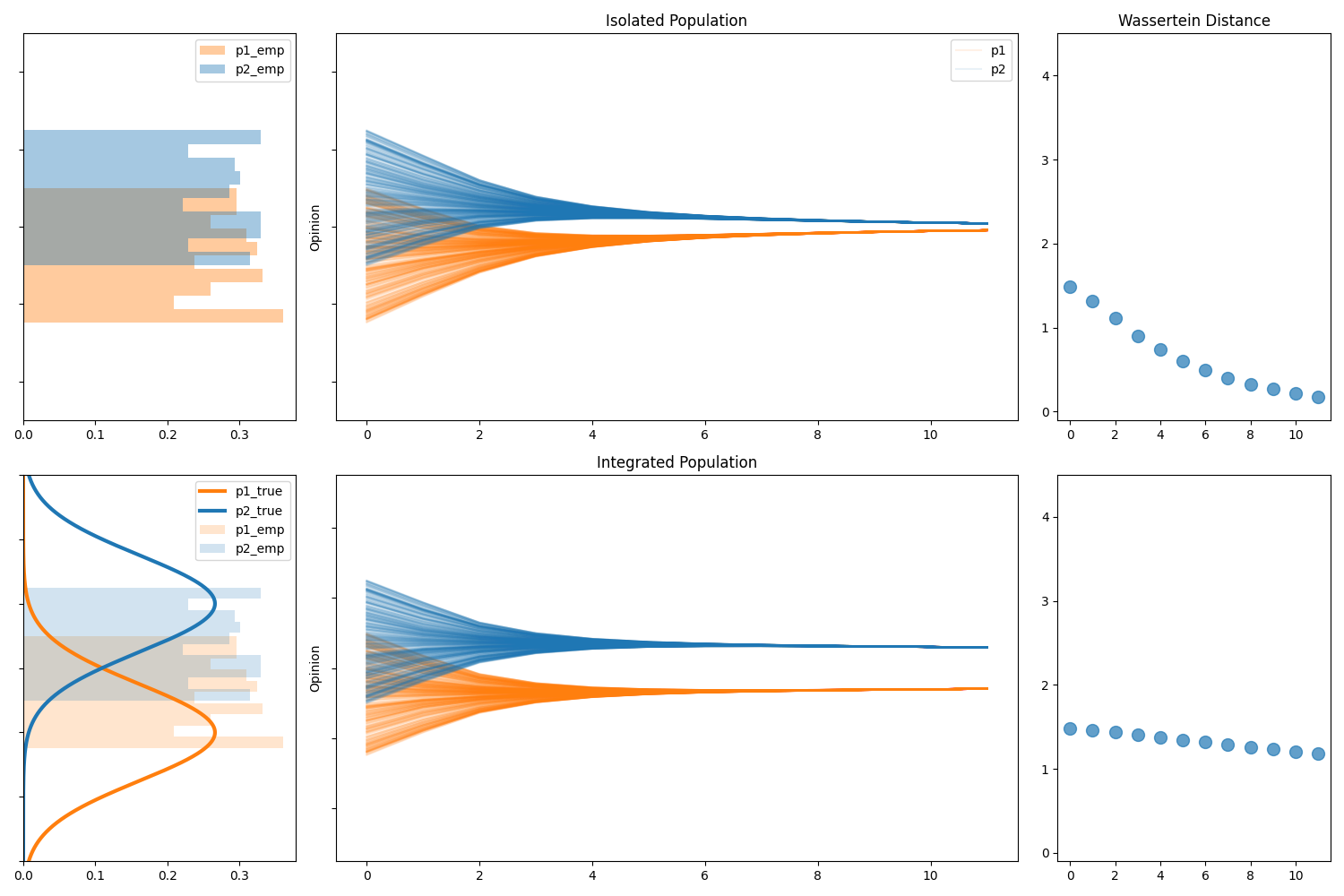}
    \caption{Isolated and Integrated Population: empirical initial opinion profiles $\mu_0^{p1_{emp}}$ and $\mu_0^{p2_{emp}}$ are sampled from uniform distribution, $N=200,\,\ \alpha=0.5,\,\ \varepsilon=0.5, \,\ \kappa_d=uniform\ kernel$, see equation in \eqref{uniform}. }
    \label{integrated}
\end{figure}
\end{example}

\section{Multi-Identity  Multi-Population Model} \label{sec:5}
In this section, we introduce a multi-identity multi-population model. Unlike the models in the previous section, here, a single agent may have multiple identities (e.g., gender, race, political orientation), and the influence of these identities on opinion formation may vary depending on the focal issue.

In the preceding sections, the whole population is assumed to be divided into unique, mutually exclusive sub-populations. So, the agents are assumed to be part of a uniquely given sub-population, and their opinions are affected by how these sub-populations interact (based on the distribution-driven dynamics). In reality, people are part of many different groups: gender, ethnic group, age, political orientation, etc. However, in general, most of the group affiliations are neglectable for a given focal issue. For example, in the case of  Russia-Ukraine war, the nationality of people and their political orientations play crucial role, and the effect of other group affiliations may be neglected. Combining identity-invariant and time-varying models, we introduce an opinion dynamics model, which considers multi-group affiliation as well. 

In our setting, a group affiliation corresponds to a single partition of the population, and multi-group affiliations corresponds to several partitions of the population. Particularly, let  $I$ be the set of agents. Consider $(M+1)$ different partitions of $I$, each of the partition corresponds to different group affiliations of agents, and let $\mathcal{M}$ is be the set of all partitions. We denote  $\Lambda_r^k$ the $k$-th sub-population of the $r$-th  partition and by $\Lambda_r(i)$, we denote the sub-population of agent $i$ for $r$-th  partition. For $r=0,1,\dots,M$, we have $\Lambda_r^k\cap\Lambda_r^j=\emptyset$ for $k\neq j$ and   
$\cup_{k=1}^{P_r}\Lambda_r^k=I$, where $P_r$ is the size of the partition or the number of sub-populations with respect to $r$ division.

\begin{equation*}
\begin{cases}
     \frac{\d X^i_{t}}{\dt}
    =-\alpha_i X^i_t +\frac{\alpha_i}{Z_M} \sum_{j=1}^{N} K(\mathcal{M})\kappa_d(X^i_{t},X^j_{t},\varepsilon_i)X^j_{t}\\
    X^{i}_0 \sim \bar{\mu}^i_0,
\end{cases}
\end{equation*}
where 
\begin{equation*}
K(\mathcal{M})=\sum_{r=0}^{M} \gamma_r Q^r(W_p(\mu^{\Lambda_r(i)}_t,\mu^{\Lambda_r(j)}_t),\sigma^r_i),\,\ Z_M=\sum_{j=1}^{N} K(\mathcal{M})\kappa_d(X^i_{t},X^j_{t},\varepsilon_i),
\end{equation*}
with $Q_r$ is being either identity invariant kernel, see Section \ref{sec:2}, or time-variant kernel, see Section \ref{sec:3}; that is,
\begin{equation*}
Q_r= \beta_{r}(t)K^r_d(W_p(\mu^{\Lambda_r(i)}_0,\mu^{\Lambda_r(j)}_0),\sigma^r_i)\,\ \text{or}\,\ Q_r=K^r_d(W_p(\mu^{\Lambda_r(i)}_t,\mu^{\Lambda_r(j)}_t),\sigma_i).
\end{equation*}
In this model, we assume that for the given focal issue, the group affiliations that have the most effect on opinion formation have bigger coefficients $\gamma$. 


\begin{example}\textbf{Group-Identity Dominance on Opinion Evolution.}
In this example, we consider multi-identity 
 multi-population model with to identity invariant kernels. 
 
Consider a population of agents that have multi-group affiliations (see Figure \ref{Partitions}). The first group affiliation gives two sub-groups that have opinions distributed neutrally but with different variations, while the second affiliation corresponds to ideology partitioning (i.e, whether or not an agent's initial opinion is larger than $0$). The group identities coming from different group affiliations affects agents' interactions differently. Here, we're using the parameter $\gamma$ to quantify the influence from different group identities: higher value of $\gamma$ implies the dominance of the second group affiliation, and vice versa.

In the first column of Figure~\ref{multi_dim}, it shows the overall flow of opinions without distinguishing between group identities. The second and third columns highlight the opinion evolution with different partitioning.

With increasing $\gamma$, the first group affiliations begin to lose their influence, depicted in the second column, we observe the emergence of opinion clusters with mixing colors (orange and blue). In contrast, the third column of Figure~\ref{multi_dim} reflects a population where ideological divisions are more pronounced. When $\gamma$ is high, indicating the dominance of the second group affiliation, agents' opinions polarize significantly around the ideology's stance. Here, opinions are sharply divided, with little overlap between the groups, illustrating the impact of strong ideological group identities on social consensus and cohesion.

\begin{figure}[h]
    \centering
\includegraphics[width=0.7\textwidth]{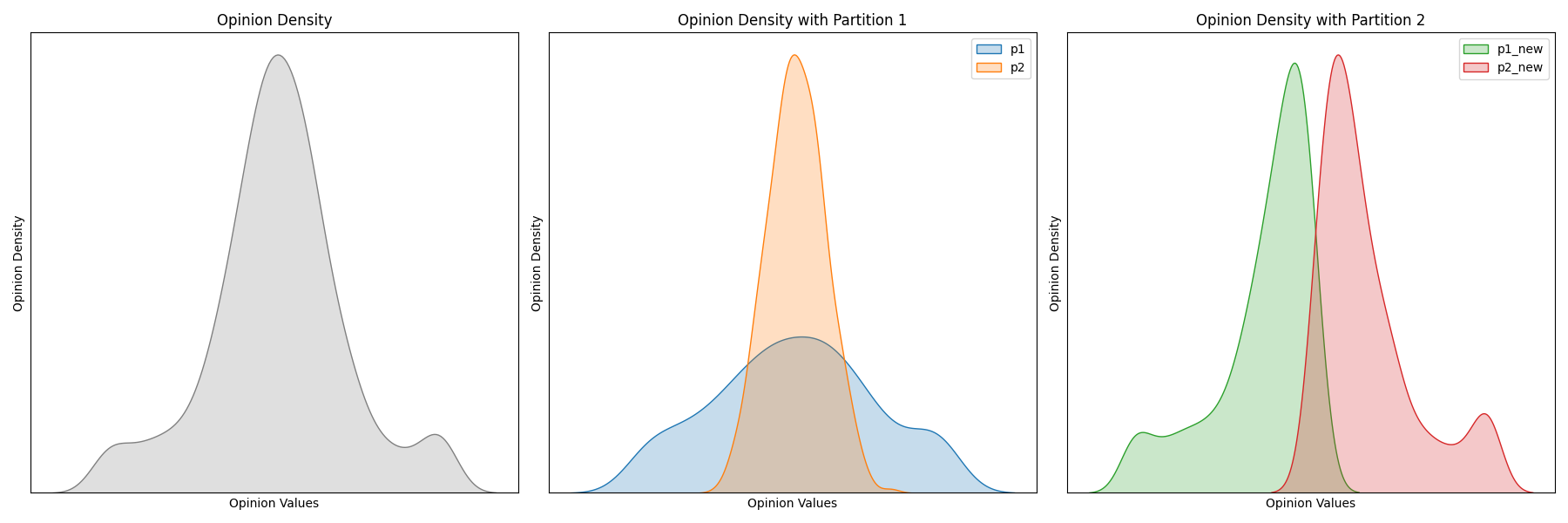}
    \caption{Opinion Distribution with Different Partitions.}
    \label{Partitions}
\end{figure}

\begin{figure}[h]
    \begin{subfigure}{0.9\textwidth}
    \centering
\includegraphics[width=0.9\textwidth]{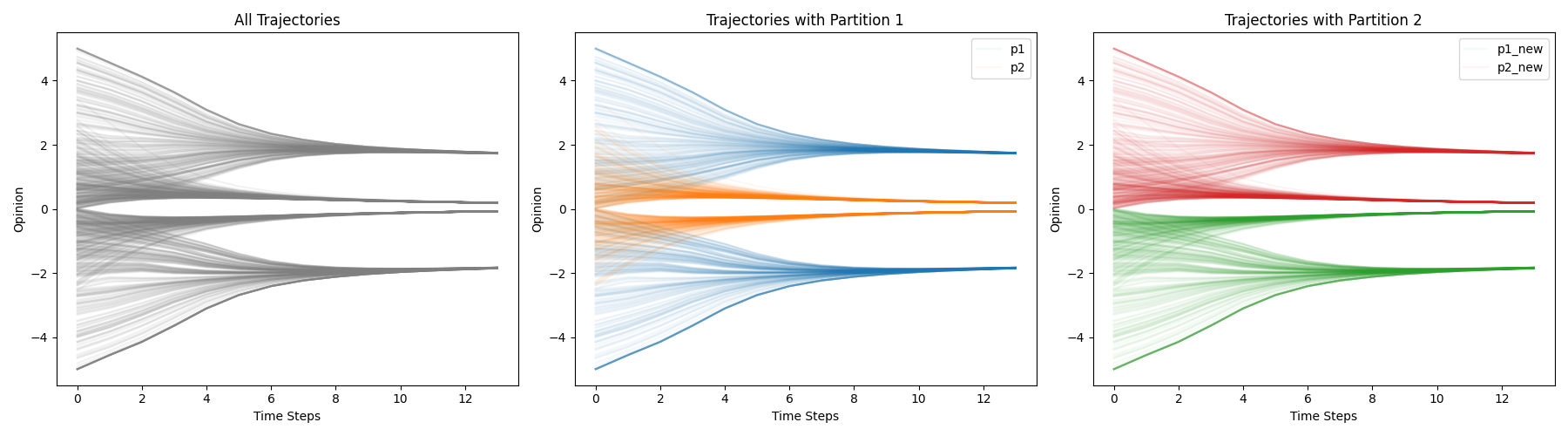}
     \caption{$\beta=0.1$}
    \end{subfigure}
    \begin{subfigure}{0.9\textwidth}
    \centering
\includegraphics[width=0.9\textwidth]{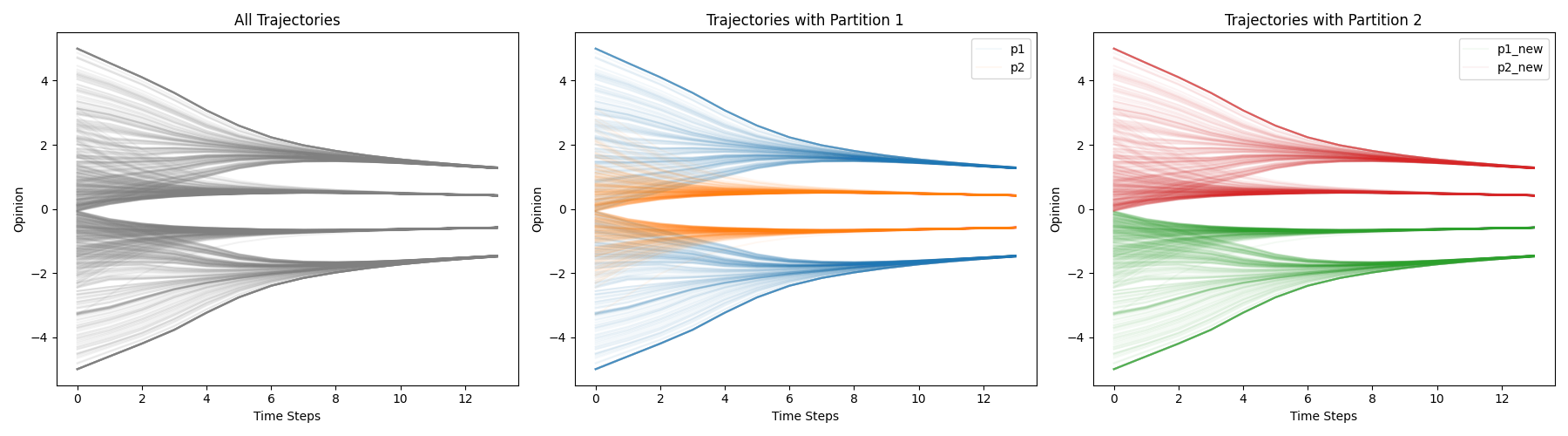}
     \caption{$\beta=0.5$}
    \end{subfigure}
    \begin{subfigure}{0.9\textwidth}
    \centering
\includegraphics[width=0.9\textwidth]{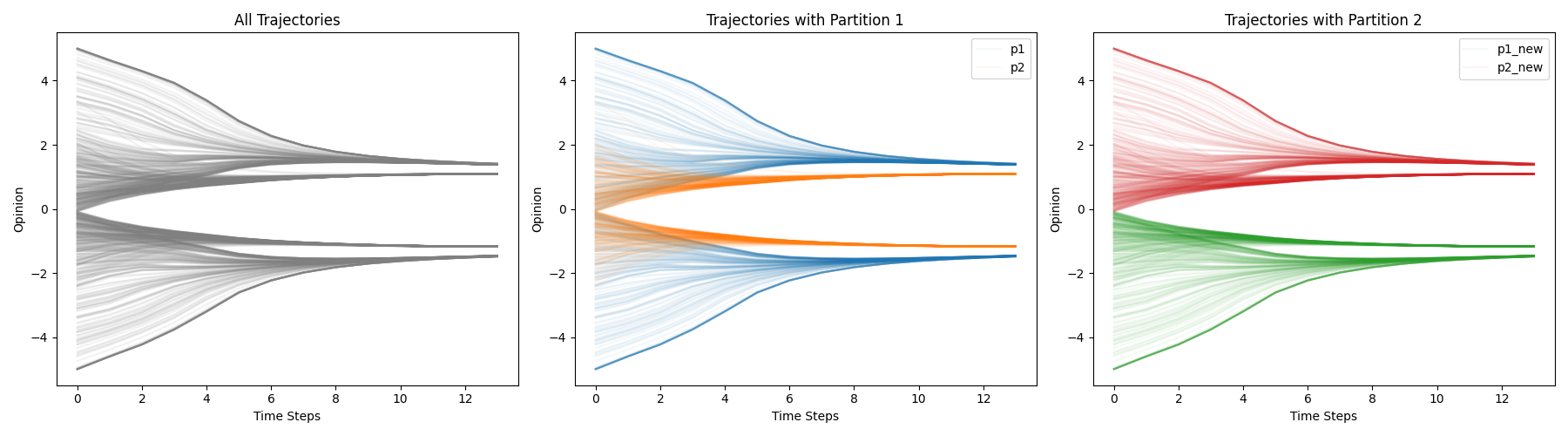}
     \caption{$\beta=0.9$}
    \end{subfigure}
    \caption{Opinion Trajectories of Different Partitions with $\gamma = 0.1,\ 0.5,\ 0.9$.}
    \label{multi_dim}
\end{figure}
\end{example}


\section{Case Study}\label{sec:4}
In this section, we present several numerical experiments using the models introduced in the preceding sections.  Each of the examples captures certain social phenomenon.

\subsection{Decaying Effects} The phenomenon of opinion dynamics with decaying effects relates to individuals gradually \textit{reverting to their original social identities after active involvement in discussions}. The study of such effects has significant implications in modelling public opinion on break out event (e.g., a significant political development, natural disaster, or technological innovation). Given that such events often generate an immediate, intense focus of public attention, it often induces a transient structural pattern in public opinion, manifesting as either consensus or polarization. However, as the fervor surrounding the event decays, these transient patterns tend to evolve towards more "stable" configurations that are more closely aligned with individuals' original social identities  (see for example \cite{fatigue}). We capture such decaying effects by adding a time-decaying factor on the population kernel:
\begin{equation*}
K(\mu^{\Lambda(i)}_0,\mu^{\Lambda(j)}_0,\sigma_i)= 
\begin{cases} \exp (-\Gamma(t)  W_p(\mu^{\Lambda(i)}_0,\mu^{\Lambda(j)}_0))
, \quad &W_p(\mu_{\mathcal{X}^i_t},\mu_{\mathcal{X}^j_t}) \leq \sigma_{i}\\
        0,   \quad & \text{otherwise}.
    \end{cases}
\end{equation*}
where $\Gamma(t)$ is a monotonically increasing function of time $t$ with $\Gamma(0) > 0$.

In particular, the population consists of three sub-groups: two extreme ($p1$ and $p2$) and one neutral ($p3$) sub-populations (see Figure \ref{decay effect}). The decaying effect is imposed on the neutral population to model its "reaction" to a break-out event. As shown in Figure \ref{decay effect}, a continuous active interaction between neutral and extreme populations lead to a break-down of the neutral population: it loses its \textit{neutrality} identity and arrives at a polarization configuration at steady state. On the other hand, when the decaying effect is imposed, a transient polarization pattern (from time step 0 to 8) is replaced by a new consensus configuration, indicating that the group reverting back to its original neutrality identity.
\begin{figure}[h]
    \centering
    \includegraphics[width=0.8\textwidth]{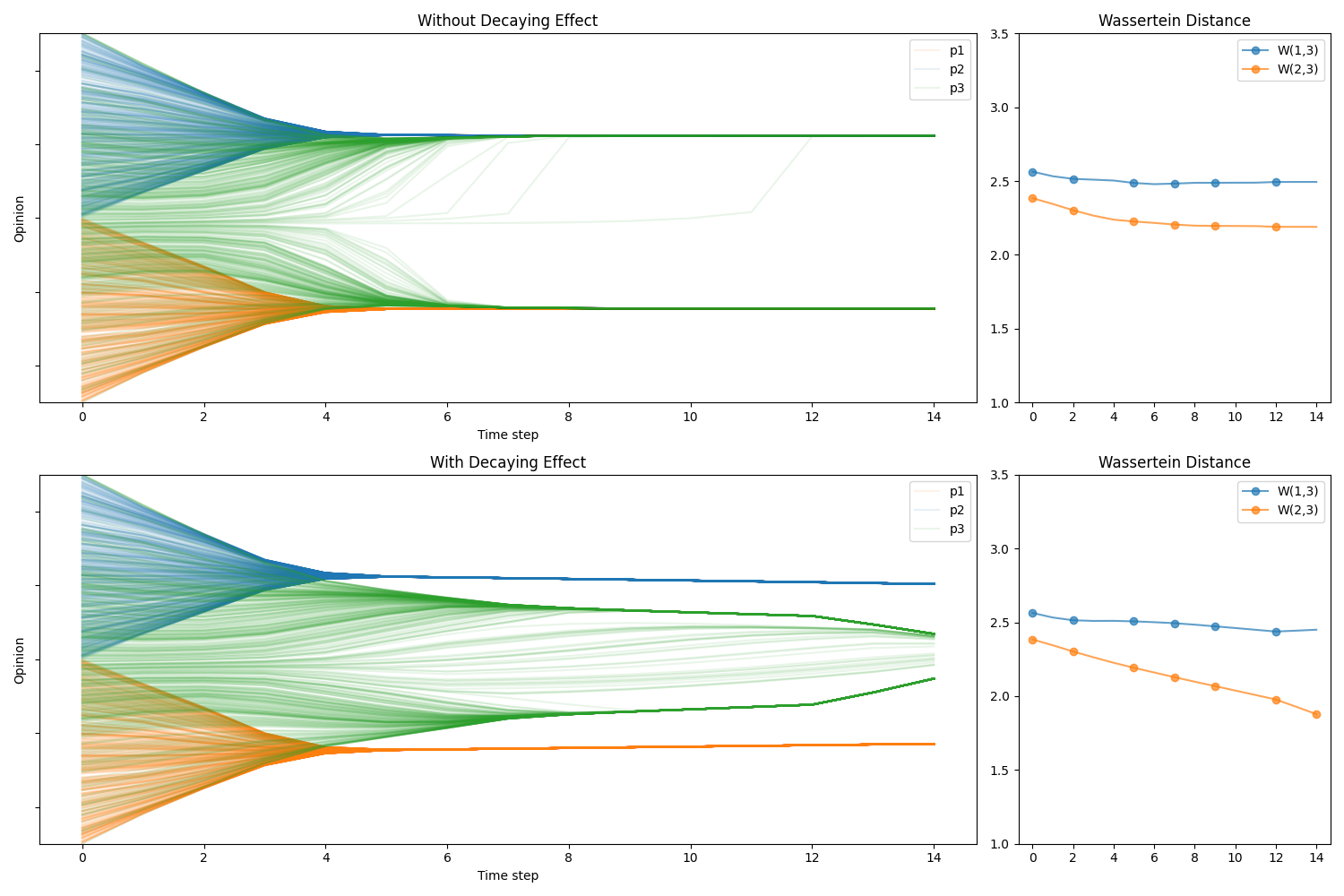}
    \caption{Decaying Effect: for $p_1$ and $p_2$, initial opinion profiles $\mu_0^{p1}$ and $\mu_0^{p2}$ are sampled from truncated Gaussian distribution and $p_3$ is sampled from uniform distribution, $N=400, \, \alpha=0.5, \, \varepsilon=0.5,\, \kappa_d=uniform\ kernel$, see equation in \eqref{uniform}. }
    \label{decay effect}
\end{figure}

\subsection{Reducing  polarization} Here, we reproduce the social experiment conducted in \cite{balietti2021reducing} by using identity invariant model (see Section \ref{sec:3}).

In \cite{balietti2021reducing} conducted social experiment about wealth distribution in US. The results show that political orientations of participants plays crucial rule of the opinion formation and leads to polarization.
The experiment demonstrated that interactions between participants with opposing views led to reduced feelings of closeness, thereby diminishing their impact on each other.
In contrast, the interactions among participants with consistent views increased closeness, therefore, the influence. The study suggests that the influence of individuals from different political parties on each other is relatively negligible compared to interactions within the same political party. However, the authors observed that introducing participants from different parties to each other in non-political matters increases the interaction between them, hence, reducing the polarization.  

We consider two populations, denoted as $p1$ (orange) and $p2$ (blue), which observe each other but do not interact due to a large group bias, as illustrated in the first row of Figure \ref{Polarization Reduction}.  To expose each other's opinions, we introduce the third population, $p3$ (green). There is no population bias between populations $p1$ and $p3$, as well as between populations $p2$ and $p3$. The representatives of population $p3$ are stubborn, meaning they influence representatives of other populations but remain unchanged themselves. In other words, this population serves as an information transferrer. So, we make populations $p1$ and $p2$ take into account each other's opinions through the population $p3$. As shown in the second row of Figure \ref{Polarization Reduction}, the information explosion reduces the polarization and may lead to the consensus. However, with a strong bias between the two populations, achieving consensus becomes unrealistic, meaning that while polarization is reduced, the populations remain polarized. We simulate this behavior by considering the threshold below kernel (see Section \ref{sec:3}), see the third row of Figure \ref{Polarization Reduction}. Specifically, when populations $p1$ and $p2$ come too close to each other, the group bias is activated, causing them to stop interacting with population $p3$. Hence, they stay polarized.
 
\begin{figure}[h]
    \centering
\includegraphics[width=0.9\textwidth]{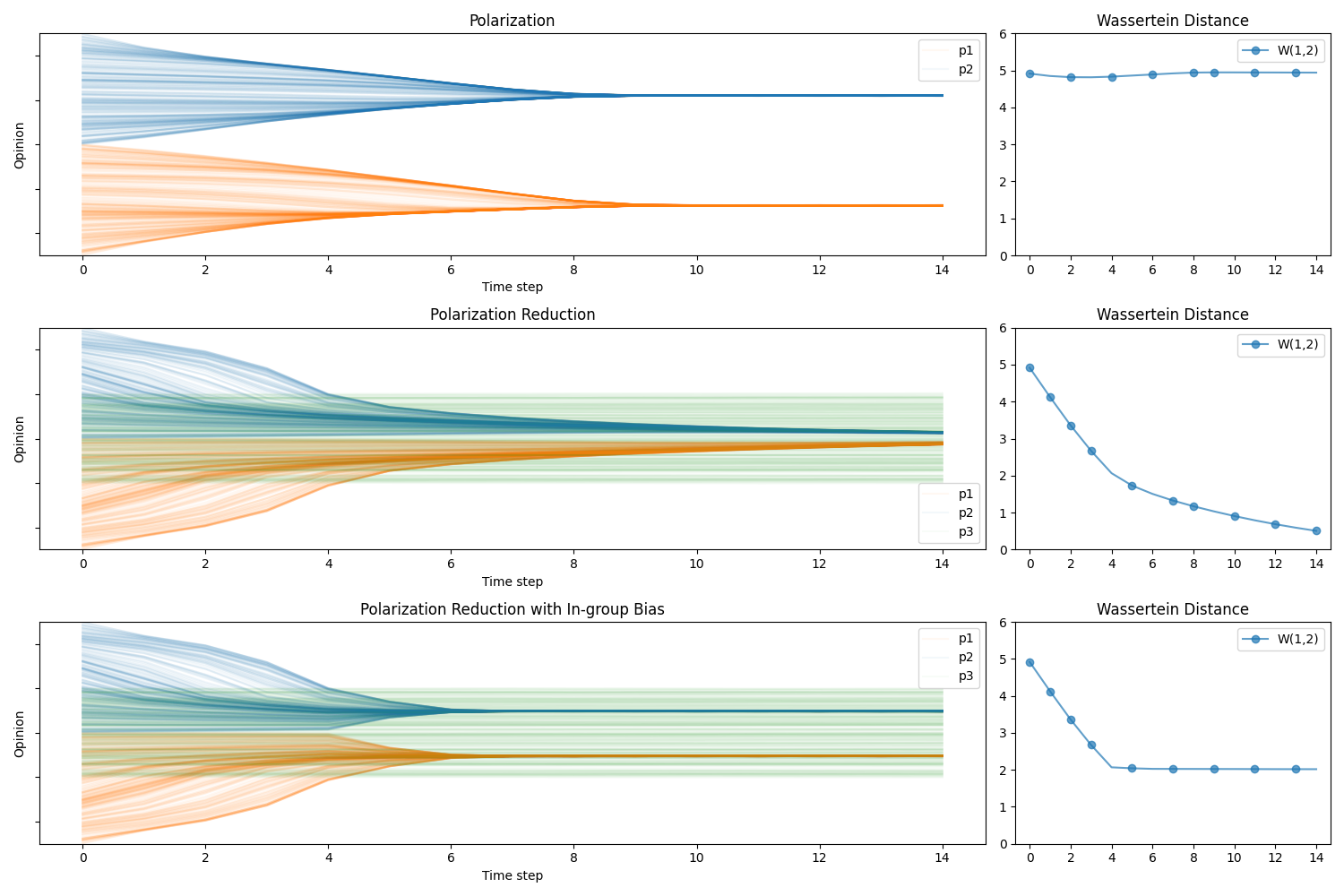}
\caption{Polarization Reduction.}
\label{Polarization Reduction}
\end{figure}




\bibliographystyle{ieeetr}
\bibliography{references}

\end{document}